\documentclass[11pt]{amsart}
\usepackage[T1]{fontenc}
\usepackage[utf8]{inputenc}
\usepackage{lmodern}
\usepackage[margin=1in]{geometry}
\usepackage{amsmath,amssymb,amsthm,mathtools}
\usepackage{tikz-cd}
\usepackage{xcolor}
\usepackage{microtype}
\usepackage[numbers,square]{natbib}
\usepackage[hidelinks]{hyperref}
\newcommand{\doi}[1]{\url{https://doi.org/#1}}
\usepackage{enumitem}
\usepackage{longtable,booktabs,array}
\usepackage{pdfcomment}

\newcommand{\Q}{\mathbb Q}

\newcommand{\Z}{\mathbb Z}

\newcommand{\Ql}{\mathbb Q_\ell}
\newcommand{\Zl}{\mathbb Z_\ell}
\newcommand{\Fl}{\mathbb F_\ell}
\newcommand{\cM}{\mathcal M}
\newcommand{\M}{\textcolor{black}{\mathcal M}}
\newcommand{\cX}{\textcolor{black}{\mathcal X}}
\newcommand{\cC}{\mathcal C}

\newcommand{\cH}{\mathcal H}
\newcommand{\cP}{\mathcal P}
\newcommand{\cG}{\mathcal G}
\newcommand{\cU}{\mathcal U}
\newcommand{\cA}{\mathcal A}
\newcommand{\p}{\mathfrak p}
\newcommand{\g}{\mathfrak g}
\renewcommand{\u}{\mathfrak u}
\newcommand{\LL}{\mathbb L}
\newcommand{\et}{\mathrm{\acute et}}
\newcommand{\cts}{\mathrm{cts}}
\newcommand{\nab}{\mathrm{nab}}
\newcommand{\geom}{\mathrm{geom}}
\DeclareMathOperator{\Gr}{Gr}
\DeclareMathOperator{\Sp}{Sp}
\DeclareMathOperator{\GSp}{GSp}

\DeclareMathOperator{\Ad}{Ad}

\DeclareMathOperator{\Lie}{Lie}
\DeclareMathOperator{\Spec}{Spec}
\DeclareMathOperator{\Sec}{Sec}
\DeclareMathOperator{\Conf}{Conf}

\DeclareMathOperator{\Hom}{Hom}
\DeclareMathOperator{\pr}{pr}
\newtheorem{theorem}{Theorem}[section]
\newtheorem*{theorem*}{Theorem}
\newtheorem{proposition}[theorem]{Proposition}
\newtheorem{lemma}[theorem]{Lemma}
\newtheorem{corollary}[theorem]{Corollary}
\theoremstyle{definition}

\theoremstyle{remark}
\newtheorem{remark}[theorem]{Remark}
\definecolor{revisionred}{RGB}{0,0,0}

\title{The section conjecture for generic pointed curves}
\author{Ma Luo}
\address{School of Mathematical Sciences, Key Laboratory of MEA
(Ministry of Education), Shanghai Key Laboratory of PMMP,
East China Normal University, Shanghai, China}
\email{mluo@math.ecnu.edu.cn}
\author{Tatsunari Watanabe}
\address{Mathematics Department, Embry-Riddle Aeronautical University,
Prescott, Arizona, USA}
\email{watanabt@erau.edu}
\thanks{Corresponding author: Tatsunari Watanabe
(\texttt{watanabt@erau.edu}).}
\thanks{The first author is supported in part by Shanghai Pilot Program for Basic Research and by Science and Technology Commission of Shanghai Municipality (No. 22DZ2229014).}
\date{}

\begin{document}

\begin{abstract}
Let $k/\mathbb Q$ be finitely generated, let $g\geq3$ and
$n\geq1$, and let $\mathcal M_{g,n/k}$ be the moduli stack
of smooth proper genus-$g$ curves with $n$ distinct ordered
markings. Put $K=k(\mathcal M_{g,n/k})$.
We prove Grothendieck's section conjecture for the generic
$n$-pointed genus-$g$ curve over $K$.
Over both the moduli stack and its function field, the
continuous sections of the arithmetic fundamental-group
extension of the universal complete curve are, up to
geometric conjugacy, exactly the $n$ tautological sections.
The same classification holds for every geometrically
pro-$\ell$ quotient, for every prime $\ell$.
Together with Hain's determination of the rational points,
these classifications give examples over fields finitely
generated over $\mathbb Q$ with any prescribed positive
number of rational points for which the full section
conjecture holds.
The proof builds on Hain's weighted-completion and
nonabelian cohomology methods, with a hyperelliptic
calculation providing the genus-three input.
An integral normalizer lemma and moving-mark arguments
pass from unipotent sections to geometrically pro-$\ell$
sections. Configuration spaces and Mochizuki's open Hom
theorem yield the full profinite classification, and
generic descent identifies the section classes over the
moduli stack with those over $K$.
\end{abstract}

\maketitle
\setcounter{tocdepth}{1}
\tableofcontents
\section{Introduction}
\label{sec:introduction}

\textcolor{black}{We prove Grothendieck's section conjecture for generic
pointed curves of genus $g\geq3$ over fields finitely
generated over $\Q$.} The result follows from a
classification of sections of the arithmetic fundamental-group
extension of the universal complete curve: the tautological
sections exhaust its section classes, both for the full
profinite geometric kernel and for each of its pro-$\ell$
quotients.

Grothendieck's section conjecture predicts that the rational
points of a smooth proper geometrically connected curve
of genus $g\geq2$ over a field finitely generated over $\Q$ are determined
by its arithmetic fundamental group. Each rational point
induces a section of the arithmetic fundamental-group
extension, well defined up to conjugation by the geometric
fundamental group. The conjecture asserts that this
construction is a bijection. For the universal curve, the
markings supply distinguished sections. The question is
whether any further group-theoretic section classes exist.

Let $k/\Q$ be finitely generated, let $g\geq3$ and $n\geq1$,
and write
\[
 \M=\cM_{g,n/k},\qquad \cX=\cC_{g,n/k},\qquad
 \pi:\cX\longrightarrow \M
\]
for the moduli stack of smooth proper genus-$g$ curves with
$n$ distinct ordered markings and its universal complete
curve. Thus $\cX$ includes the images of the tautological
sections $s_1,\ldots,s_n$; the additional-marking space
$\cM_{g,n+1/k}$ is their complement.
Put $K=k(\M)$ and $X_K=\cX\times_{\M}\Spec K$, and let
$x_1,\ldots,x_n\in X_K(K)$ be the marked points.
\textcolor{black}{The field $K$ is itself finitely generated
over $\Q$.}
\textcolor{black}{Throughout the paper, $\pi_1$ denotes the \'etale
fundamental group.}
Choose compatible geometric basepoints and put
\[
 \Delta=\textcolor{black}{\pi_1}(\cX_{\bar\eta}),\qquad
 E=\textcolor{black}{\pi_1}(\cX),\qquad
 \Gamma=\textcolor{black}{\pi_1}(\M).
\]
For a prime $\ell$, set
$N=\Delta^{(\ell)}$ and
$E_\ell=E/\ker(\Delta\to N)$.
We consider the exact sequences
\[
 1\longrightarrow\Delta\longrightarrow E
   \longrightarrow\Gamma\longrightarrow1,
 \qquad
 1\longrightarrow N\longrightarrow E_\ell
   \longrightarrow\Gamma\longrightarrow1.
\]
In the second sequence only the geometric kernel has been
replaced by its maximal pro-$\ell$ quotient.
Write $\Sec_{\M}(\Delta)$ and $\Sec_{\M}(N)$ for their continuous
section classes modulo conjugation by the indicated kernels.
The notation $\Sec_K(\Delta)$ and $\Sec_K(N)$ refers to the
corresponding extensions for $X_K/K$.

\begin{theorem*}[Main theorem]
Let $k/\Q$ be finitely generated, $g\geq3$, and $n\geq1$.
Every class in $\Sec_{\M}(\Delta)$ is represented by exactly
one of $s_1,\ldots,s_n$. For every prime $\ell$, the natural
maps in the commutative diagram
\[
\begin{tikzcd}[column sep=large]
 \Sec_{\M}(\Delta) \arrow[r] \arrow[d,"\mathrm{res}"'] &
 \Sec_{\M}(N) \arrow[d,"\mathrm{res}"] \\
 \Sec_K(\Delta) \arrow[r] & \Sec_K(N)
\end{tikzcd}
\]
are bijections. Each of these four sets consists of the
$n$ distinct tautological classes.
Moreover, the natural maps
\[
 X_K(K)\longrightarrow\Sec_K(\Delta)
 \longrightarrow\Sec_K(N)
\]
are bijections. In particular, Grothendieck's section
conjecture holds for $X_K/K$.
\end{theorem*}

The theorem combines Theorems~\ref{thm:proell-classification}
and~\ref{thm:profinite-classification} with
Corollary~\ref{cor:generic-section-conjecture}.
No level structure is required on $\M$; extensions to full
symplectic level covers and certain topologically ample
families are given in Remark~\ref{rem:ample-families}.

\textcolor{black}{To the authors' knowledge, prior to the
recent announcement of a proof of the $p$-adic section
conjecture by OpenAI~\cite{OpenAIPadic}, no examples had
been established of smooth proper hyperbolic curves
over fields finitely generated over $\Q$ with a nonempty
set of rational points for which the full section
conjecture holds.
Our results provide such examples in every genus
$g\geq3$ and with any prescribed positive number
$n$ of rational points.}

Hain proved that $X_K(K)=\{x_1,\ldots,x_n\}$ for $g\geq3$
\textcolor{black}{\cite[\textcolor{black}{Thm.}~1, p.~711]{HainRational}. He also established the}
unipotent section conjecture for the generic pointed curve
when $g\geq5$, under the hypotheses of his Theorem~3, and
proved nonsplitting of the full fundamental-group extension
for the unmarked generic curve when $g\geq5$
\cite[\textcolor{black}{Thms.}~2--3]{HainRational}.
Our main result classifies the full profinite and
geometrically pro-$\ell$ section classes in the pointed
case, including genera three and four. The rational-point count is Hain's.
For the section classifications, we first use his methods,
together with the genus-three input described below, to
classify unipotent sections over the moduli stack for
$g\geq3$. We then establish the geometrically pro-$\ell$
and full profinite classifications over the moduli stack
and use generic descent to obtain the corresponding
classifications over $K$.

\textcolor{black}{The geometrically pro-$\ell$ section
conjecture fails in general: Hoshi constructed proper
hyperbolic curves over number fields admitting geometrically
pro-$p$ sections that do not arise from rational points,
including examples with infinitely many section classes
\cite[Thms.~A and~B]{HoshiNongeometric}.
For the generic pointed curves considered here, however,
every geometrically pro-$\ell$ section of $X_K/K$ arises
from a marked point, for every prime $\ell$.}

We use weighted completion and Hain's graded description
of nonabelian cohomology
\cite{HainMatsumoto,HainRational,HainNab} to compute the
unipotent section classes over the moduli stack.
Hain's geometric calculation determines the possible
weight-$-1$ maps for $g\geq4$; in genus three, restriction
to the hyperelliptic locus excludes the remaining candidate
\cite{WatanabeHyper2019,WatanabeHyper2025,LuoWatanabeBirman}.
His cup-product obstruction eliminates the higher-weight
arithmetic classes, leaving exactly the $n$ tautological
classes.

To pass to pro-$\ell$ conjugacy, we prove an integral
normalizer lemma. Let $N$ be a pro-$\ell$ surface group
of genus at least two, and let $P$ be the group of
$\Ql$-points of its unipotent completion. The natural
map $N\to P$ is a closed topological embedding, and
\[
 u^{-1}Nu\subseteq N\quad\Longrightarrow\quad u\in N
 \qquad(u\in P).
\]
The proof uses the integral lower central series and
the triviality of the center of the associated graded
Lie algebra modulo $\ell$.
Moving the $j$th marking realizes every inner automorphism
of $N$, so a unipotent conjugator relating a section to
$s_j$ satisfies this condition. This gives the pro-$\ell$
classification for every prime $\ell$.

For the full profinite classification, we pull back to
the configuration space of a fixed curve and extend across
the collision diagonals. An all-prime cohomological argument
along the diagonal supplies the openness required by
Mochizuki's open Hom theorem
\cite[\textcolor{black}{Thm.}~A]{MochizukiLocal}, which then lifts conjugacy
from a single geometrically pro-$\ell$ quotient to the full
fundamental group. Normality of the configuration subgroup
and the triviality of the geometric kernel's center
give the corresponding conjugacy over $\M$.

Finally, good reduction and purity identify section classes
over $\M$ with those over its generic point. We follow
Stix's specialization method \cite{StixSections} and its
use in \cite{WatanabePositive}, treating both full profinite
and geometrically pro-$\ell$ sections.

\textcolor{black}{The classifications also hold over full level covers and over
smooth bases whose classifying morphisms induce isomorphisms on
geometric fundamental groups with such a cover. This includes
Hain's generic ample linear sections of dimension at least two
\cite[\textcolor{black}{Prop.}~11.3]{HainRational}. These are examples of his
topologically ample families; Remark~\ref{rem:ample-families}
specifies the hypothesis used for this extension.}

Sections~\ref{sec:conventions}
and~\ref{sec:weighted-completion} introduce the fundamental
groups and completions. Section~\ref{sec:integral-conjugation}
proves the integral normalizer lemma.
Sections~\ref{sec:moving} and~\ref{sec:unipotent-sections}
treat moving markings and the unipotent classification.
\textcolor{black}{Section~\ref{sec:descent} establishes
generic descent. Section~\ref{sec:full} develops the
ingredients for the full profinite classification,
and Section~\ref{sec:main-proofs} proves the main
results and their consequences.}

\section{Fundamental groups and section conventions}
\label{sec:conventions}

\textcolor{black}{We fix the fundamental-group extensions and conjugacy
conventions used throughout the paper.}

\subsection{The universal curve and fundamental-group extensions}
\label{subsec:universal-extensions}

Let $k$ be a field finitely generated over $\Q$, and let
$g\geq2$ and $n\geq0$. Write
\[
 \M=\cM_{g,n/k},\qquad \cX=\cC_{g,n/k},\qquad
 \pi:\cX\longrightarrow \M
\]
for the moduli stack of smooth proper genus-$g$ curves with
$n$ distinct ordered markings and its universal complete curve.
For $n\geq1$, the markings define tautological sections
$s_1,\ldots,s_n$ of $\pi$. The additional-marking space
$\cM_{g,n+1/k}$ is the open substack of $\cX$ obtained by removing
their images; when $n=0$, it is the whole of $\cX$.

Fix an algebraic closure $\bar k$ of $k$ and a geometric point
$\bar\eta$ of $\M_{\bar k}$. Let $C=\cX_{\bar\eta}$ be the
corresponding geometric fiber, and choose a geometric point
$\bar x$ of $C$. Put
\[
 \Delta=\textcolor{black}{\pi_1}(C,\bar x),\qquad
 E=\textcolor{black}{\pi_1}(\cX,\bar x),\qquad
 \Gamma=\textcolor{black}{\pi_1}(\M,\bar\eta).
\]
All homomorphisms between profinite groups are understood to
be continuous.

\begin{lemma}\label{lem:universal-homotopy}
For every $n\geq0$, the universal complete curve induces a
short exact sequence
\begin{equation}\label{eq:profinite-extension}
 1\longrightarrow\Delta\longrightarrow E
 \xrightarrow{\pi_*}\Gamma\longrightarrow1.
\end{equation}
The corresponding sequence after base change to $\bar k$
is also exact.
\end{lemma}

\begin{proof}
The morphism $\pi$ is representable, smooth and proper,
with geometrically connected fibers. Its homotopy sequence
is therefore exact at $E$ and $\Gamma$
\cite[Secs.~A.1 and~A.4, pp.~99 and~102]{NoohiFundamental}.

For injectivity on the left, put $k'=k(\mu_3)$ and choose
a primitive cube root of unity. The corresponding symplectic
level-$3$ cover
\[
 T=\cM_{g,n/k'}[3]\longrightarrow \M
\]
is finite \'etale, and $T$ is a smooth, geometrically connected
fine moduli scheme over $k'$
\cite[\textcolor{black}{Sec.}~5.1.2, p.~719]{HainRational}.
With compatible basepoints, Hain's fiber-injectivity lemma
\cite[\textcolor{black}{Lem.}~4.1, p.~716]{HainRational} applied to
$X_T=\cX\times_{\M} T\to T$ gives
\[
 \Delta\hookrightarrow\textcolor{black}{\pi_1}(X_T)\hookrightarrow E.
\]
The second map is injective because $X_T\to \cX$ is a connected
finite \'etale cover. Their composite is the original fiber
homomorphism, proving its injectivity. The same argument
applies after base change to $\bar k$.
\end{proof}

Write
\[
 \textcolor{black}{\Gamma^{\geom}}=\textcolor{black}{\pi_1}(\M_{\bar k},\bar\eta),
 \qquad G_k=\operatorname{Gal}(\bar k/k).
\]
The arithmetic fundamental-group sequence of the base is
\begin{equation}\label{eq:arithmetic-base-extension}
 1\longrightarrow\textcolor{black}{\Gamma^{\geom}}\longrightarrow\Gamma
 \longrightarrow G_k\longrightarrow1;
\end{equation}
see \cite[\textcolor{black}{Sec.}~5, p.~717]{HainRational}.
\subsection{Geometrically pro-\texorpdfstring{$\ell$}{ell}
and prounipotent kernels}
\label{subsec:geometric-kernels}

Fix a prime $\ell$. Let
\[
 N=\Delta^{(\ell)},\qquad
 R_\ell=\ker(\Delta\longrightarrow N),\qquad
 E_\ell=E/R_\ell,
\]
where $\Delta^{(\ell)}$ denotes the maximal pro-$\ell$
quotient of $\Delta$. The subgroup $R_\ell$ is closed and
characteristic in $\Delta$, hence normal in $E$.
Quotienting \eqref{eq:profinite-extension} by $R_\ell$ gives
the geometrically pro-$\ell$ extension
\begin{equation}\label{eq:proell-extension}
 1\longrightarrow N\longrightarrow E_\ell
 \longrightarrow\Gamma\longrightarrow1.
\end{equation}
Only the geometric kernel has been replaced by its pro-$\ell$
quotient; the quotient group remains $\Gamma$.

Let $\cP$ be the continuous unipotent completion of $\Delta$
over $\Ql$, and put
\[
 P=\cP(\Ql),\qquad \p=\Lie\cP.
\]
If $\cP=\varprojlim_\alpha U_\alpha$ is expressed as an
inverse limit of finite-dimensional unipotent $\Ql$-groups,
we give
\[
 P=\varprojlim_\alpha U_\alpha(\Ql)
\]
the inverse-limit topology, with each $U_\alpha(\Ql)$ carrying
its $\ell$-adic topology. All homomorphisms to $P$ are understood
to be continuous for this topology.

The completion homomorphism factors uniquely as
\[
 \Delta\longrightarrow N\xrightarrow{i}P.
\]
Indeed, its image in each $U_\alpha(\Ql)$ is compact and
therefore pro-$\ell$, so each quotient map factors continuously
through $N$. These factorizations are compatible with the
transition maps. Thus $\cP$ is also the continuous unipotent
completion of $N$; see
\cite[\textcolor{black}{App.}~A.2, pp.~159--160]{HainMatsumoto}.
Under the characteristic-zero comparison with a compact
genus-$g$ surface, this agrees with the unipotent completion
used in Subsection~\ref{subsec:integral-lcs}, by
\cite[\textcolor{black}{Thm.}~A.6, p.~161]{HainMatsumoto}.
Proposition~\ref{prop:intersection} will show that $i$ is a
closed topological embedding.

Section~\ref{sec:weighted-completion} constructs the prounipotent extension
\begin{equation}\label{eq:unipotent-extension}
 1\longrightarrow P\longrightarrow E_P
 \longrightarrow\Gamma\longrightarrow1
\end{equation}
by pulling back the weighted-completion extension, following
\cite[\textcolor{black}{Prop.}~15.1]{HainRational}.
It also constructs the continuous comparison homomorphism
$\Phi:E_\ell\to E_P$, with
\[
\begin{tikzcd}[column sep=small]
1 \arrow[r] & N \arrow[r] \arrow[d,"i"] &
E_\ell \arrow[r] \arrow[d,"\Phi"] &
\Gamma \arrow[r] \arrow[d,equal] & 1 \\
1 \arrow[r] & P \arrow[r] &
E_P \arrow[r] & \Gamma \arrow[r] & 1.
\end{tikzcd}
\]
The topology on $E_P$ is the pullback topology from that
construction. Neither extension requires the choice of a
section.

\subsection{Section classes and comparison maps}
\label{subsec:section-classes}

For an extension of topological groups
$1\to A\to B\to\Gamma\to1$, a section is a continuous
homomorphism $s:\Gamma\to B$ whose composite with
$B\to\Gamma$ is the identity. We identify two sections when
they are conjugate by an element of the kernel $A$.
With this convention, write
\[
\begin{aligned}
 \Sec_{\M}(\Delta)&=
 \{\text{continuous sections of }E\to\Gamma\}/\Delta,\\
 \Sec_{\M}(N)&=
 \{\text{continuous sections of }E_\ell\to\Gamma\}/N,\\
 \Sec_{\M}(P)&=
 \{\text{continuous sections of }E_P\to\Gamma\}/P.
\end{aligned}
\]
The quotient homomorphism and $\Phi$ induce comparison maps
\begin{equation}\label{eq:section-comparisons}
 \Sec_{\M}(\Delta)\longrightarrow\Sec_{\M}(N)
 \xrightarrow{\Phi_{\M}}\Sec_{\M}(P).
\end{equation}
We also write
\[
 H^1_{\nab}(\Gamma,\cP)(\Ql)=\Sec_{\M}(P)
\]
for the nonabelian cohomology set attached to
\eqref{eq:unipotent-extension}. This notation refers to the
specified extension and does not presuppose a splitting.

An algebraic section of $\cX\to \M$ over $k$ induces a class in
$\Sec_{\M}(\Delta)$: choosing a path in the geometric fiber
identifies its value at $\bar\eta$ with the chosen basepoint,
and changing this path changes the induced section by
$\Delta$-conjugacy. We use the same notation for an algebraic
section and a chosen induced section of the fundamental-group
extension. In particular, the tautological sections give
classes in all three sets in \eqref{eq:section-comparisons}.

The subscript $\M$ indicates sections over the moduli stack.
For $g\geq3$, its generic point is $\Spec K$, where $K=k(\M)$,
since the generic curve has no automorphisms;
see \cite[\textcolor{black}{Prop.}~5.1]{HainRational}.
Sections over this generic point concern the extensions
pulled back along $G_K\to\Gamma$, with compatible basepoints, where
$G_K=\operatorname{Gal}(\bar K/K)$, and will be denoted by
$\Sec_K(\Delta)$, $\Sec_K(N)$ and $\Sec_K(P)$.
The relation between these sets and the corresponding sets
over $\M$ is treated in Section~\ref{sec:descent}.

\subsection{Difference cocycles and conjugacy}
\label{subsec:difference-cocycles}

\textcolor{black}{To compare a section with a tautological one, we express
their difference as a cocycle.}
Assume $n\geq1$ and fix $j\in\{1,\ldots,n\}$.
Choose the section $s_j:\Gamma\to E_\ell$ induced by the
$j$th marking, and put $s_{j,P}=\Phi\circ s_j$.
Conjugation by these sections defines actions, both denoted
by $\alpha_j$, on $N$ and $P$:
\[
\begin{aligned}
 \alpha_j(\gamma)(h)&=s_j(\gamma)h s_j(\gamma)^{-1}
     &&(h\in N),\\
 \alpha_j(\gamma)(p)&=s_{j,P}(\gamma)p s_{j,P}(\gamma)^{-1}
     &&(p\in P).
\end{aligned}
\]
Since $\Phi|_N=i$, these actions satisfy
\[
 i\bigl(\alpha_j(\gamma)(h)\bigr)
 =\alpha_j(\gamma)\bigl(i(h)\bigr).
\]

If $s:\Gamma\to E_\ell$ is another section, define
\[
 z(\gamma)=s(\gamma)s_j(\gamma)^{-1}.
\]
The two factors have the same image in $\Gamma$, so
$z(\gamma)\in N$. The homomorphism property of $s$ gives
the cocycle identity
\begin{equation}\label{eq:difference-cocycle}
 z(\gamma\textcolor{black}{\gamma'})
 =z(\gamma)\alpha_j(\gamma)\bigl(z(\textcolor{black}{\gamma'})\bigr)
 \qquad(\gamma,\textcolor{black}{\gamma'}\in\Gamma).
\end{equation}
Conversely, every continuous map $z:\Gamma\to N$ satisfying
\eqref{eq:difference-cocycle} defines a section by
$s(\gamma)=z(\gamma)s_j(\gamma)$.

Our conjugacy convention is
$s^v(\gamma)=v^{-1}s(\gamma)v$ for $v\in N$.
Keeping the reference section $s_j$ fixed, the cocycle
associated with $s^v$ is
\[
\begin{aligned}
 z^v(\gamma)
 &=s^v(\gamma)s_j(\gamma)^{-1}\\
 &=v^{-1}z(\gamma)s_j(\gamma)v\,s_j(\gamma)^{-1}\\
 &=v^{-1}z(\gamma)\alpha_j(\gamma)(v),
\end{aligned}
\]
where we used $s(\gamma)=z(\gamma)s_j(\gamma)$ and the
definition of $\alpha_j$.
Thus $[s]=[s_j]$ in $\Sec_{\M}(N)$ if and only if
$s(\gamma)=v^{-1}s_j(\gamma)v$ for one element $v\in N$
and every $\gamma\in\Gamma$. Equivalently,
\[
 z(\gamma)=v^{-1}\alpha_j(\gamma)(v)
 \qquad\text{for every }\gamma\in\Gamma.
\]
The same conventions apply with $P$ and $s_{j,P}$ in place
of $N$ and $s_j$. In particular,
$\Phi_{\M}([s])=[s_{j,P}]$ if and only if there is one
element $u\in P$ such that
\begin{equation}\label{eq:comparison-coboundary}
 i\bigl(z(\gamma)\bigr)
 =u^{-1}\alpha_j(\gamma)(u)
 \qquad\text{for every }\gamma\in\Gamma.
\end{equation}
Indeed, this identity is equivalent to
$\Phi(s(\gamma))=u^{-1}s_{j,P}(\gamma)u$.
Section~\ref{sec:moving} will show that, in this situation,
the conjugating element can be taken in $i(N)$.
\section{Weighted completion and graded sections}
\label{sec:weighted-completion}

Retain the notation of Section~\ref{sec:conventions}.
Thus $k/\Q$ is finitely generated, $g\geq2$, $n\geq0$,
and $\ell$ is a prime. We recall the completion constructions
and section comparison used below, following
\cite{HainMatsumoto,HainRational,HainNab}.
The surface-group input used here is proved in
Section~\ref{sec:integral-conjugation}.

\subsection{Monodromy and weighted completion}
\label{subsec:weighted-monodromy}

Put
\[
 H=H^1_{\et}(C,\Ql(1))\cong H_1(\p),
 \qquad R=\GSp(H).
\]
\textcolor{black}{The cup product gives the alternating pairing
$\langle-,-\rangle:\Lambda^2H\to\Ql(1)$.} The composition of the monodromy representation
$\rho_{\M}:\Gamma\to R(\Ql)$ with the similitude character
$R(\Ql)\to\Ql^\times$ is the $\ell$-adic cyclotomic character. Its restriction to
$\textcolor{black}{\Gamma^{\geom}}$ has Zariski-dense image in $\Sp(H)$.
Moreover, the cyclotomic character of $k$ has infinite image:
the algebraic closure of $\Q$ in $k$ is a number field,
and $k$ is regular over that field.
Consequently $\rho_{\M}$ is Zariski dense in $R$;
see \cite[\textcolor{black}{Sec.}~8, p.~728]{HainRational}.
We use the central cocharacter
\[
 \omega:\mathbb G_m\longrightarrow R,
 \qquad t\longmapsto t^{-1}\operatorname{id}_H.
\]
Thus $H$ has weight $-1$ and $\Ql(1)$ has weight $-2$.

An extension of $R$ by a prounipotent group is
\emph{negatively weighted} if the weights of the induced
$\omega$-action on the abelianization of its unipotent
radical are negative. The weighted completion of
$(\Gamma,\rho_{\M},\omega)$ is an extension
\[
 1\longrightarrow\cU_{\M}\longrightarrow\cG_{\M}
 \longrightarrow R\longrightarrow1
\]
together with a continuous Zariski-dense lift
$\widetilde\rho_{\M}:\Gamma\to\cG_{\M}(\Ql)$ of $\rho_{\M}$.
It is universal among continuous lifts of $\rho_{\M}$ to
negatively weighted extensions of $R$.
Define $\cG_{\cX}$ similarly using
$\rho_{\cX}=\rho_{\M}\circ\pi_*:E\to R(\Ql)$, and write
$\widetilde\rho_{\cX}:E\to\cG_{\cX}(\Ql)$ for its completion map.
For $T=\M,\cX$, put
\[
 \g_T=\Lie\cG_T,\qquad \u_T=\Lie\cU_T.
\]
These definitions are those of
\cite[\textcolor{black}{Def.}~4.1]{HainMatsumoto} and
\cite[\textcolor{black}{Sec.}~6.4]{HainRational}.

Every finite-dimensional representation of a negatively
weighted extension has a natural weight filtration.
It can be constructed by choosing a lift of $\omega$ and
summing its eigenspaces of weights at most the given index;
the filtration is independent of the lift.
Morphisms are strict for this filtration, and
$\Gr^W_\bullet$ is exact. These statements extend to
pro-objects; see \cite[\textcolor{black}{Thm.}~3.12]{HainMatsumoto}.
In particular,
\[
 W_0\g_T=\g_T,\qquad W_{-1}\g_T=\u_T,
 \qquad \Gr^W_0\g_T=\Lie R,
\]
and the associated graded Lie algebras carry natural
$R$-actions. Throughout, Lie algebra homomorphisms and
graded sections are continuous morphisms of the relevant
pro-objects.

\textcolor{black}{We will also use the cohomological comparison
of Hain--Matsumoto
\cite[\textcolor{black}{Thms.}~4.6 and~4.9]{HainMatsumoto}.}
For a finite-dimensional
$R$-module $V$ of negative weight, it gives
\begin{equation}\label{eq:weighted-cohomology-comparison}
 \begin{aligned}
 H^1(\u_{\M},V)^R&\xrightarrow{\sim}H^1(\Gamma,V),\\
 H^2(\u_{\M},V)^R&\lhook\joinrel\longrightarrow H^2(\Gamma,V).
 \end{aligned}
\end{equation}
Here $\u_{\M}$ acts trivially on $V$, Lie algebra cohomology
is continuous, and the right-hand side is continuous
cohomology of the profinite group.

\subsection{Weighted and geometric completion sequences}
\label{subsec:weighted-curve-sequences}

We compare the weighted completions of the arithmetic
fundamental groups with the relative completions of their
geometric kernels. We first collect the properties of
weighted completion used below.

\begin{proposition}[{\cite[\textcolor{black}{Thms.}~3.9 and~3.12,
\textcolor{black}{Props.}~4.13--4.14]{HainMatsumoto}}]
\label{prop:weighted-properties}
\leavevmode
\begin{enumerate}
\renewcommand{\labelenumi}{\textup{(\roman{enumi})}}

\item
Weighted completion is functorial for continuous
homomorphisms compatible with the reductive representations
and central cocharacters.

\item
Every finite-dimensional representation $V$ of a negatively
weighted extension has a canonical weight filtration.
The action on $\Gr^W_mV$ factors through the reductive
quotient and has weight $m$.
Every equivariant homomorphism $f:V\to V'$ is strict:
\[
 f(W_mV)=f(V)\cap W_mV'.
\]
Thus $\Gr^W_\bullet$ is exact.
These properties extend to projective limits of
finite-dimensional representations.

\item
Given an exact sequence of profinite groups
\[
 1\longrightarrow\Pi_0\longrightarrow\widetilde\Pi
 \longrightarrow\Pi\longrightarrow1,
\]
let $\widetilde{\mathcal G}$ and $\mathcal G$ be the
weighted completions for a continuous Zariski-dense
representation of $\Pi$ and its pullback, using the
same central cocharacter. Then
\[
 (\Pi_0)^{\mathrm{un}}_{/\Ql}
 \longrightarrow\widetilde{\mathcal G}
 \longrightarrow\mathcal G\longrightarrow1
\]
is exact, where $(\Pi_0)^{\mathrm{un}}_{/\Ql}$ is the
continuous unipotent completion of $\Pi_0$.
\end{enumerate}
\end{proposition}

The third assertion uses the quotient argument of
Hain--Matsumoto's Prop.~4.14:
the image of $(\Pi_0)^{\mathrm{un}}_{/\Ql}$ is normal
in $\widetilde{\mathcal G}$, and the quotient has the
universal property of $\mathcal G$.

Functoriality gives the homomorphism
$\cG_{\cX}\to\cG_{\M}$. We claim that its kernel is precisely
the continuous unipotent completion $\cP$ of the
geometric fiber:
\begin{equation}\label{eq:weighted-curve-extension}
 1\longrightarrow\cP\longrightarrow\cG_{\cX}
 \longrightarrow\cG_{\M}\longrightarrow1.
\end{equation}
Right exactness follows from
Proposition~\ref{prop:weighted-properties}.
For injectivity, we use Hain's exactness criterion
\cite[\textcolor{black}{Prop.}~6.11, p.~725]{HainRational}.
The action on $H_1(\cP)=H$ factors through $R$, and
$H$ is finite-dimensional and pure of weight $-1$.

Moreover, $\cP$ has trivial center for $g\geq2$.
Indeed, write $L^\bullet$ for the closed lower central
series of $\cP$ and $\p$, as recalled in
Subsection~\ref{subsec:integral-lcs}.
The first nonzero lower central class of a nonzero
central element of $\p$ would be central in
$\Gr_L^\bullet\p$, contrary to
Lemma~\ref{lem:lie-center} and the surface Lie algebra
presentation. Logarithm gives the corresponding statement
for $\cP$. Hain's criterion therefore applies to
\eqref{eq:profinite-extension}.
Since the criterion concerns profinite groups, it applies
directly to the fundamental groups of our stacks.

The induced weight filtration on $\p$ satisfies
\begin{equation}\label{eq:weight-lower-central-fiber}
 W_{-r}\p=L^r\p\qquad(r\geq1).
\end{equation}
To see this, the bracket induces a surjection
$H^{\otimes r}\to\Gr_L^r\p$, so $\Gr_L^r\p$ is pure
of weight $-r$.
On each finite-dimensional quotient $\p/L^m\p$,
the lower central filtration therefore has the defining
properties of the weight filtration.
Passing to the inverse limit proves the identity;
compare \cite[\textcolor{black}{Rem.}~8.5, p.~730]{HainRational}.

On $P=\cP(\Ql)$, we use the closed lower central series
\[
 L^1P=P,\qquad L^{r+1}P=\overline{[P,L^rP]},
\]
with closure taken in the inverse-limit $\ell$-adic topology.
Then
\[
 L^rP=(L^r\cP)(\Ql)=\exp(L^r\p).
\]
Indeed, logarithm and exponential identify the lower
central series on finite-dimensional unipotent quotients;
passing to the inverse limit and taking closures gives
the assertion.

Taking Lie algebras in \eqref{eq:weighted-curve-extension}
and applying strictness gives
\begin{equation}\label{eq:weighted-graded-extension}
 0\longrightarrow\Gr^W_\bullet\p
 \longrightarrow\Gr^W_\bullet\g_{\cX}
 \longrightarrow\Gr^W_\bullet\g_{\M}\longrightarrow0.
\end{equation}

We now introduce the geometric completions.
For $T=\M,\cX$, the monodromy representation
\[
 \textcolor{black}{\pi_1}(T_{\bar k})\longrightarrow\Sp(H)(\Ql)
\]
is continuous and Zariski dense.
Its \emph{continuous relative completion} is a
proalgebraic extension
\[
 1\longrightarrow\cU_T^{\geom}
 \longrightarrow\cG_T^{\geom}
 \longrightarrow\Sp(H)\longrightarrow1
\]
with prounipotent kernel, together with a continuous
Zariski-dense lift
\[
 \textcolor{black}{\pi_1}(T_{\bar k})
 \longrightarrow\cG_T^{\geom}(\Ql).
\]
It is universal among continuous lifts of geometric
monodromy to prounipotent extensions of $\Sp(H)$;
see \cite[\textcolor{black}{Sec.}~6.2, pp.~720--721]{HainRational}.
No negative-weight condition is imposed in this
construction. The weight filtrations on these geometric
completions will instead come from the arithmetic action.

Put
\[
 \g_T^{\geom}=\Lie\cG_T^{\geom},
 \qquad
 \u_T^{\geom}=\Lie\cU_T^{\geom}.
\]
The inverse image of $\Sp(H)$ in $\cG_T$ is a
prounipotent extension of $\Sp(H)$, and the restriction
of the arithmetic completion map to
$\textcolor{black}{\pi_1}(T_{\bar k})$ takes values in this inverse image.
The universal property therefore gives a natural
homomorphism $\cG_T^{\geom}\to\cG_T$.
These homomorphisms fit into the commutative diagram
\[
\begin{tikzcd}
 \textcolor{black}{\pi_1}(T_{\bar k}) \arrow[r] \arrow[d] &
 \textcolor{black}{\pi_1}(T) \arrow[d] \\
 \cG_T^{\geom}(\Ql) \arrow[r] &
 \cG_T(\Ql).
\end{tikzcd}
\]

The geometric fundamental-group extension similarly gives
\begin{equation}\label{eq:geometric-completion-extension}
 1\longrightarrow\cP\longrightarrow\cG_{\cX}^{\geom}
 \longrightarrow\cG_{\M}^{\geom}\longrightarrow1.
\end{equation}
Continuous relative completion is right exact by the
same quotient argument as in
Proposition~\ref{prop:weighted-properties}.
The map from $\cP$ is injective because its composite
with $\cG_{\cX}^{\geom}\to\cG_{\cX}$ is the injection in
\eqref{eq:weighted-curve-extension}.
This is the argument of
\cite[\textcolor{black}{Prop.}~7.1, pp.~726--727]{HainRational}.

To equip the geometric sequence with weights, we use
the arithmetic action.
By \cite[\textcolor{black}{Prop.}~8.4, p.~730]{HainRational},
the conjugation action of $\Gamma$ on the geometric
fundamental group induces a $\cG_{\M}$-action on
$\g_{\M}^{\geom}$ and $\u_{\M}^{\geom}$, making them
pro-objects of the category of $\cG_{\M}$-modules.
The natural map $\g_{\M}^{\geom}\to\g_{\M}$ is equivariant.

The two completion sequences give the diagram
\begin{equation}\label{eq:geometric-arithmetic-comparison}
\begin{tikzcd}[column sep=small]
0 \arrow[r] & \p \arrow[r] \arrow[d,equal] &
\g_{\cX}^{\geom} \arrow[r] \arrow[d] &
\g_{\M}^{\geom} \arrow[r] \arrow[d] & 0 \\
0 \arrow[r] & \p \arrow[r] &
\g_{\cX} \arrow[r] & \g_{\M} \arrow[r] & 0.
\end{tikzcd}
\end{equation}
Both rows are exact and the map on kernels is the
identity, so the right-hand square is Cartesian:
\[
 \g_{\cX}^{\geom}
 \cong\g_{\cX}\times_{\g_{\M}}\g_{\M}^{\geom}.
\]
Give $\g_{\cX}$ its adjoint action and
$\g_{\M}^{\geom}$ the action obtained through
$\cG_{\cX}\to\cG_{\M}$.
Equivariance of the maps to $\g_{\M}$ makes this fiber
product a pro-$\cG_{\cX}$-module.
Thus the geometric sequence has compatible weight
filtrations, and its associated graded sequence is
exact and $R$-equivariant.

For $g\geq3$, we have
$W_{-1}\g_{\M}^{\geom}=\u_{\M}^{\geom}$
by \cite[\textcolor{black}{Rem.}~8.5, p.~730]{HainRational}.
Since $\p$ has only negative weights, the geometric
extension gives the same identity for $\cX$.
Consequently,
\[
 0\longrightarrow\Gr^W_\bullet\p
 \longrightarrow\Gr^W_\bullet\u_{\cX}^{\geom}
 \longrightarrow\Gr^W_\bullet\u_{\M}^{\geom}
 \longrightarrow0
\]
is the exact sequence of graded Lie algebras used in
Subsection~\ref{subsec:geometric-graded-sections}.

Finally, let $\cA_k$ be the weighted completion of $G_k$
with respect to the cyclotomic character and the
cocharacter $t\mapsto t^{-2}$.
The arithmetic fundamental-group sequence induces
\[
 \cG_{\M}^{\geom}\longrightarrow\cG_{\M}
 \longrightarrow\cA_k\longrightarrow1,
\]
which is exact by
\cite[\textcolor{black}{Prop.}~8.1, p.~729]{HainRational}.
In particular, the geometric image is normal.
The natural map $\u_{\M}^{\geom}\to\u_{\M}$ is
$\cG_{\M}$-equivariant, so its image is stable under $\g_{\M}$.
Strictness identifies the associated graded of this
image with the image of the associated graded map;
the latter is therefore a graded ideal.
No injectivity of $\cG_{\M}^{\geom}\to\cG_{\M}$ is required.
\subsection{Rational points of the completions and the pullback extension}
\label{subsec:weighted-pullback}

\textcolor{black}{To obtain an extension of topological groups over $\Gamma$,
we check exactness on $\Ql$-points.
Since the two completions have the same reductive quotient
$R$, \eqref{eq:weighted-curve-extension} gives a surjection
$\u_{\cX}\to\u_{\M}$ of pronilpotent Lie algebras.
Under the duality between pro-vector spaces and ordinary
vector spaces, this becomes an injection. Splitting that injection and dualizing gives
a pro-linear right inverse. Thus the original map is
surjective on inverse-limit vector spaces, and functorial
logarithm and exponential give a surjection}
$\cU_{\cX}(\Ql)\to\cU_{\M}(\Ql)$.

Levi's theorem gives surjectivity of
$\cG_T(\Ql)\to R(\Ql)$ for $T=\M,\cX$;
see \cite[Sec.~6.5]{HainRational}.
Thus both rows in the diagram
\[
\begin{tikzcd}[column sep=small]
1 \arrow[r] & \cU_{\cX}(\Ql) \arrow[r] \arrow[d,two heads] &
\cG_{\cX}(\Ql) \arrow[r] \arrow[d] &
R(\Ql) \arrow[r] \arrow[d,equal] & 1 \\
1 \arrow[r] & \cU_{\M}(\Ql) \arrow[r] &
\cG_{\M}(\Ql) \arrow[r] & R(\Ql) \arrow[r] & 1
\end{tikzcd}
\]
are exact, and the left vertical map is surjective by
the preceding argument. A diagram chase shows that
$\cG_{\cX}(\Ql)\to\cG_{\M}(\Ql)$ is surjective.
Since its kernel is $\cP(\Ql)=P$, we obtain the exact sequence
\begin{equation}\label{eq:weighted-points-extension}
 1\longrightarrow P\longrightarrow\cG_{\cX}(\Ql)
 \longrightarrow\cG_{\M}(\Ql)\longrightarrow1.
\end{equation}
This is the rational-point argument underlying
\cite[Prop.~15.1]{HainRational}.

The continuous linear right inverse $\u_{\M}\to\u_{\cX}$
obtained by duality, together with logarithm, exponential,
and compatible Levi splittings, gives a continuous
set-theoretic right inverse to
$\cG_{\cX}(\Ql)\to\cG_{\M}(\Ql)$.
This surjective homomorphism is therefore open.

Give rational points of proalgebraic groups their
inverse-limit $\ell$-adic topologies, and define
\begin{equation}\label{eq:EP-pullback-definition}
 E_P=\cG_{\cX}(\Ql)\times_{\cG_{\M}(\Ql)}\Gamma,
\end{equation}
where the map from $\Gamma$ is $\widetilde\rho_{\M}$.
The pullback topology and \eqref{eq:weighted-points-extension}
give the extension \eqref{eq:unipotent-extension}, with
kernel $P$ carrying the topology specified in Section~\ref{sec:conventions}.
The homomorphism
\[
 E\longrightarrow E_P,
 \qquad \textcolor{black}{x}\longmapsto
 \bigl(\widetilde\rho_{\cX}(\textcolor{black}{x}),
       \pi_*(\textcolor{black}{x})\bigr)
\]
restricts on $\Delta$ to its continuous unipotent completion
map. It therefore kills $R_\ell$ and descends continuously
to
\[
 \Phi:E_\ell\longrightarrow E_P.
\]
It restricts to $i:N\to P$ on kernels and induces the
identity on $\Gamma$. This construction uses no reference
section and applies also when $n=0$.

\subsection{Nonabelian cohomology and graded sections}
\label{subsec:completion-section-comparison}

We describe the comparison between continuous sections,
sections of weighted completions, and graded Lie algebra
sections used below. For a more detailed treatment,
see \cite{HainRational,HainNab}.

Recall from \eqref{eq:weight-lower-central-fiber} that
$W_{-m}\p=L^m\p$ for $m\geq1$. Write
\[
 W_{-m}\cP:=\exp(W_{-m}\p)=L^m\cP,
 \qquad
 W_{-m}P:=(W_{-m}\cP)(\Ql)\textcolor{black}{=L^mP}.
\]
For $r\geq1$, put
\[
 \cP_r=\cP/W_{-r-1}\cP,
 \qquad
 P_r=\cP_r(\Ql),
 \qquad
 \widehat{\cG}_r=\cG_{\cX}/W_{-r-1}\cP.
\]
Since $W_{-r-1}\cP$ is characteristic in $\cP$, it is
normal in $\cG_{\cX}$, and
\eqref{eq:weighted-curve-extension} induces an extension
\[
 1\longrightarrow\cP_r
 \longrightarrow\widehat{\cG}_r
 \longrightarrow\cG_{\M}
 \longrightarrow1.
\]
The group $\widehat{\cG}_r$ remains a negatively weighted
extension of $R$. Its kernel $\cP_r$ is finite-dimensional
and unipotent, since its Lie algebra is a quotient of the
free nilpotent Lie algebra of class $r$ on $H$.

By Subsection~\ref{subsec:weighted-pullback}, the associated
sequence of $\Ql$-points is exact. Define
\[
 E_r=\widehat{\cG}_r(\Ql)
       \times_{\cG_{\M}(\Ql)}\Gamma.
\]
With the pullback topology, this gives a commutative
diagram with exact rows:
\[
\begin{tikzcd}[column sep=small]
1 \arrow[r] & P_r \arrow[r] \arrow[d,equal] &
E_r \arrow[r] \arrow[d] &
\Gamma \arrow[r] \arrow[d,"\widetilde\rho_{\M}"] & 1 \\
1 \arrow[r] & P_r \arrow[r] &
\widehat{\cG}_r(\Ql) \arrow[r] &
\cG_{\M}(\Ql) \arrow[r] & 1.
\end{tikzcd}
\]
The right-hand square is Cartesian. The quotient
identifications
\[
 P_r\cong P/W_{-r-1}P,
 \qquad
 E_r\cong E_P/W_{-r-1}P
\]
hold as topological groups by the surjectivity and
openness argument of
Subsection~\ref{subsec:weighted-pullback}.

Let $H^1_{\nab}(\cG_{\M},\cP_r)$ denote Hain's
nonabelian-cohomology ind-affine scheme associated to
$\widehat{\cG}_r\to\cG_{\M}$.
Its $\Ql$-points are the $P_r$-conjugacy classes of
proalgebraic sections
\cite[\textcolor{black}{Thm.}~4.6, p.~593]{HainNab}.
Write $H^1_{\nab}(\Gamma,\cP_r)(\Ql)$ for the
$P_r$-conjugacy classes of continuous sections of
$E_r\to\Gamma$.
\textcolor{black}{We use analogous notation for the full extensions.
These section sets require no reference section; the
comparison below uses Hain's description of rational
points and makes no finite-type assertion.}

\begin{proposition}[Comparison of section sets]
\label{prop:completion-graded-section-comparison}
For every $r\geq1$, there are natural bijections
\begin{equation}\label{eq:completion-graded-section-comparison}
\begin{aligned}
 H^1_{\nab}(\Gamma,\cP_r)(\Ql)
 &\cong H^1_{\nab}(\cG_{\M},\cP_r)(\Ql)\\
 &\cong
 \left\{\begin{array}{c}
 \text{$R$-equivariant graded Lie algebra sections of}\\
 \Gr^W_\bullet(\g_{\cX}/W_{-r-1}\p)
       \longrightarrow\Gr^W_\bullet\g_{\M}
 \end{array}\right\},
\end{aligned}
\end{equation}
where graded maps are understood in the category of
pro-vector spaces.
These bijections are compatible with truncation.
Moreover,
\begin{equation}\label{eq:inverse-limit-section-classes}
 H^1_{\nab}(\Gamma,\cP)(\Ql)
 \cong\varprojlim_r H^1_{\nab}(\Gamma,\cP_r)(\Ql),
\end{equation}
and the analogous comparison holds for the full kernel
$\cP$, with graded map
$\Gr^W_\bullet\g_{\cX}\to\Gr^W_\bullet\g_{\M}$.
\end{proposition}

\begin{proof}
A continuous section of $E_r\to\Gamma$ determines a
continuous lift
$t:\Gamma\to\widehat{\cG}_r(\Ql)$ of
$\widetilde\rho_{\M}$.
The universal property of weighted completion extends
$t$ uniquely to a proalgebraic homomorphism
$\sigma:\cG_{\M}\to\widehat{\cG}_r$.
Its composite with the projection to $\cG_{\M}$ is the
identity, since it agrees with the identity on the
Zariski-dense image of $\Gamma$.
Restriction along $\widetilde\rho_{\M}$ is the inverse
construction, and both constructions respect
$P_r$-conjugacy. This proves the first bijection, as in
\cite[\textcolor{black}{Sec.}~5, p.~596]{HainNab}.
The second is Hain's graded-section description
\cite[\textcolor{black}{Cor.}~4.7, p.~593]{HainNab}, given by
$[\sigma]\mapsto\Gr^W_\bullet d\sigma$.
Both comparisons commute with truncation.

For the inverse-limit assertion, we work with proalgebraic
sections, using the same universal-property comparison for
the full extension. Choose a Levi section $R\to\cG_{\cX}$,
giving compatible cocharacter lifts $\widehat\omega_r$
and $\widetilde\omega$.
By \cite[\textcolor{black}{Lem.}~4.2 and \textcolor{black}{Cor.}~4.5,
pp.~590--593]{HainNab}, each finite-level class has a
unique normalized representative $\sigma_r$ satisfying
\[
 \sigma_r\circ\widetilde\omega=\widehat\omega_r,
\]
with a unique normalizing conjugator.
The projection of $\sigma_{r+1}$ is normalized and
represents the same class as $\sigma_r$, so uniqueness
makes these representatives compatible.

Every finite-dimensional algebraic quotient of $\cG_{\cX}$
has nilpotent image of $\cP$, and therefore kills
$W_{-r-1}\cP=L^{r+1}\cP$ for sufficiently large $r$.
Hence
\[
 \cG_{\cX}\cong\varprojlim_r\widehat{\cG}_r,
\]
and compatible normalized representatives define a full
section. This proves surjectivity in
\eqref{eq:inverse-limit-section-classes}.

For injectivity, let $s$ and $t$ be full sections whose
finite-level classes agree. Their truncations $s_r,t_r$
have a common normalized representative $\sigma_r$, so
there are unique $u_r,\textcolor{black}{\widetilde u_r}\in P_r$ satisfying
\[
 u_rs_ru_r^{-1}=\sigma_r,
 \qquad
 \textcolor{black}{\widetilde u_r}t_r\textcolor{black}{\widetilde u_r}^{-1}=\sigma_r.
\]
Projection preserves these identities. By uniqueness,
the transition map $P_{r+1}\to P_r$ sends $u_{r+1}$ to
$u_r$ and $\textcolor{black}{\widetilde u_{r+1}}$ to $\textcolor{black}{\widetilde u_r}$. Hence there are elements
$u,\textcolor{black}{\widetilde u}\in P=\varprojlim_r P_r$ whose images in $P_r$ are
$u_r,\textcolor{black}{\widetilde u_r}$, respectively, for every $r$.
Taking inverse limits gives $usu^{-1}=\textcolor{black}{\widetilde u} t\textcolor{black}{\widetilde u}^{-1}$,
so $s$ and $t$ are conjugate by $\textcolor{black}{\widetilde u}^{-1}u\in P$.
This proves injectivity.

Finally, for each fixed weight $m$, the natural map
\[
 \Gr^W_m\g_{\cX}\longrightarrow
 \Gr^W_m(\g_{\cX}/W_{-r-1}\p)
\]
is an isomorphism whenever $-r-1<m$.
A compatible family of finite-level graded sections
therefore determines a unique full graded map, weight
by weight. It is $R$-equivariant and a right inverse
because these properties hold at every finite level.
Bracket preservation follows by choosing a level at
which the two input weights and their sum have stabilized.
This establishes the full graded comparison.
\end{proof}

\section{Integral surface groups and conjugation}
\label{sec:integral-conjugation}

We prove that an element $u\in P$ satisfying
$u^{-1}Nu\subseteq N$ belongs to $N$. This will allow us
to recover pro-$\ell$ conjugacy from unipotent conjugacy
in Section~\ref{sec:moving}.

\subsection{Surface groups and the integral lower central series}
\label{subsec:integral-lcs}

\textcolor{black}{Let $g\geq2$, and let $\Pi_g$ be the discrete
surface group of genus $g$, with presentation}
\[
 \Pi_g=\left\langle a_1,b_1,\ldots,a_g,b_g\ \middle|\
       \prod_{j=1}^g[a_j,b_j]=1\right\rangle,
 \qquad [x,y]=xyx^{-1}y^{-1}.
\]
Fix a prime $\ell$, and put
\[
 N=\Pi_g^{(\ell)},\qquad
 \cP=(\Pi_g)^{\mathrm{un}}_{/\Ql},\qquad
 P=\cP(\Ql),\qquad \p=\Lie\cP.
\]
\textcolor{black}{Using the comparison of
Subsection~\ref{subsec:geometric-kernels}, identify these
with the corresponding objects of Section~\ref{sec:conventions}.}
Write $\eta:\Pi_g\to N$ and $\rho:\Pi_g\to P$ for the completion
homomorphisms. 
\textcolor{black}{We retain the topology on $P$ specified in
Subsection~\ref{subsec:geometric-kernels}.}

\textcolor{black}{We use $L^\bullet$ for the ordinary lower central series
of $\Pi_g$ and the closed lower central series of $N$:}
\[
 L^1\Pi_g=\Pi_g,\qquad L^{r+1}\Pi_g=[\Pi_g,L^r\Pi_g],
 \qquad
 L^1N=N,\qquad L^{r+1}N=\overline{[N,L^rN]}.
\]
\textcolor{black}{Brackets denote commutator subgroups, and closures
in $N$ are taken in its pro-$\ell$ topology.}
\textcolor{black}{For $\cP$ and $\p$, use the closed lower central
series of Subsection~\ref{subsec:weighted-curve-sequences},
so that}
{\color{black}
\[
 L^r\p=\Lie(L^r\cP),\qquad
 L^rP=(L^r\cP)(\Ql)=\exp(L^r\p).
\]
}
\textcolor{black}{Put $Q_r=L^r\Pi_g/L^{r+1}\Pi_g$, and write
$\Gr_L^\bullet$ for the associated graded Lie algebras.}

\textcolor{black}{For a commutative ring $A$, write $\LL_A$
for the free graded Lie algebra over $A$ on the indicated
degree-one generators. In the following presentation,
$(\theta)$ denotes the Lie ideal generated by $\theta$.} \textcolor{black}{Labute's integral calculation \cite[\textcolor{black}{Thm.}, p.~17]{Labute} gives a natural}
graded Lie algebra isomorphism
\begin{equation}\label{eq:labute}
 \Gr_L^\bullet\Pi_g
 \cong
 \frac{\LL_{\Z}(A_1,B_1,\ldots,A_g,B_g)}{(\theta)},
 \qquad
 \theta=\sum_{j=1}^g[A_j,B_j],
\end{equation}
\textcolor{black}{and every $Q_r$ is free abelian of finite rank.}
\textcolor{black}{Here $A_j$ and $B_j$ are the classes of $a_j$
and $b_j$, and $\theta$ is the primitive initial form of the
surface relator.}

\begin{proposition}[Integral lower central series]
\label{prop:intersection}
The homomorphism $\rho$ extends uniquely to a continuous
homomorphism $i:N\to P$, which is a closed topological embedding.
For every $r\geq1$, the completion homomorphisms identify
$\Gr_L^r i$ with
\[
 \Gr_L^rN
 \cong Q_r\otimes_{\Z}\Zl
 \hookrightarrow Q_r\otimes_{\Z}\Ql
 \cong \Gr_L^rP.
\]
Identifying $N$ with its image in $P$, we have
\[
 N\cap L^rP=L^rN
 \qquad(r\geq1).
\]
\end{proposition}

\begin{proof}
\textcolor{black}{By \cite[\textcolor{black}{Lem.}~A.7, p.~161]{HainMatsumoto},
$\rho$ extends uniquely to a continuous homomorphism
$i:N\to P$. Since $N$ is compact and $P$ is Hausdorff,
$i(N)$ is closed. Continuity and compatibility with
commutators, together with closedness of the lower central
terms of $P$, give
\begin{equation}\label{eq:filtration-inclusion}
 i(L^rN)\subseteq L^rP \qquad(r\geq1).
\end{equation}}

\textcolor{black}{Put
\[
 \cP_{\Q}=(\Pi_g)^{\mathrm{un}}_{/\Q},
 \qquad \p_{\Q}=\Lie\cP_{\Q},
\]
and let $\rho_{\Q}:\Pi_g\to\cP_{\Q}(\Q)$ be the completion
homomorphism. For each $r\geq1$, define
\[
 \lambda_{r,\Q}:Q_r\otimes_{\Z}\Q
 \longrightarrow\Gr_L^r\p_{\Q},
 \qquad
 \bar x\otimes t\longmapsto
 t\log\rho_{\Q}(x)\bmod L^{r+1}\p_{\Q},
\]
where $x\in L^r\Pi_g$.
The Baker--Campbell--Hausdorff formula makes these maps
a graded Lie algebra homomorphism whose degree-one part
is the canonical identification
$H_1(\Pi_g,\Q)\cong\p_{\Q}/L^2\p_{\Q}$.
By \eqref{eq:labute} and
\cite[\textcolor{black}{Prop.}~9.2, p.~735]{HainRational}, both graded
Lie algebras have the natural presentation
$\LL_{\Q}(A_1,B_1,\ldots,A_g,B_g)/(\theta)$.
Since $\lambda_{\bullet,\Q}$ fixes the degree-one
generators, it is an isomorphism.}

\textcolor{black}{For each $m\geq1$, the universal property
identifies the unipotent completion of $\Pi_g/L^m\Pi_g$
with $\cP_{\Q}/L^m\cP_{\Q}$ over $\Q$ and with
$\cP/L^m\cP$ over $\Ql$.
These quotients are finite-dimensional: their Lie algebras
are quotients of a free nilpotent Lie algebra on finitely
many generators.
By \cite[\textcolor{black}{Prop.}~A.1, p.~158]{HainMatsumoto}, there are
compatible isomorphisms
\[
 \cP/L^m\cP
 \cong
 (\cP_{\Q}/L^m\cP_{\Q})\times_{\Q}\Ql.
\]
Taking Lie algebras and kernels of successive quotient
maps gives
\[
 (\Gr_L^r\p_{\Q})\otimes_{\Q}\Ql
 \xrightarrow{\sim}\Gr_L^r\p.
\]
Compatibility with the completion homomorphisms therefore
gives the isomorphism
\[
 \lambda_{r,\Ql}:Q_r\otimes_{\Z}\Ql
 \xrightarrow{\sim}\Gr_L^r\p,
 \qquad
 \bar x\otimes t\longmapsto
 t\log\rho(x)\bmod L^{r+1}\p.
\]
}

\textcolor{black}{Since $L^rP=\exp(L^r\p)$, the
Baker--Campbell--Hausdorff formula gives the isomorphism
\begin{equation}\label{eq:rational-comparison}
 \textcolor{black}{\chi_r}:
 Q_r\otimes_{\Z}\Ql
 \xrightarrow[\lambda_{r,\Ql}]{\sim}
 \Gr_L^r\p
 \xrightarrow[\overline{\exp}]{\sim}
 L^rP/L^{r+1}P.
\end{equation}
This is a topological isomorphism for the usual $\ell$-adic
topology on the source and the quotient topology on the
target: lifting a basis gives a continuous linear section
of $L^r\p\to\Gr_L^r\p$, and logarithm and exponential
are continuous. Explicitly,
\[
 \textcolor{black}{\chi_r}(\bar x\otimes t)
 =\exp\bigl(t\log\rho(x)\bigr)\bmod L^{r+1}P.
\]
}

\textcolor{black}{Density of $\eta(\Pi_g)$ in $N$ and
continuity of commutators give, by induction on $r$,
\begin{equation}\label{eq:closure-discrete}
 \overline{\eta(L^r\Pi_g)}=L^rN.
\end{equation}}

\textcolor{black}{The induced map $Q_r\to L^rN/L^{r+1}N$
has dense image. Since its target is an abelian
pro-$\ell$ group and $Q_r$ is free abelian of finite rank,
it extends to a continuous $\Zl$-linear homomorphism
\[
 q_r:Q_r\otimes_{\Z}\Zl\longrightarrow L^rN/L^{r+1}N.
\]
Its image is compact and dense, so $q_r$ is surjective.}
By \eqref{eq:filtration-inclusion}, $i$ induces a continuous
homomorphism
\[
 j_r:L^rN/L^{r+1}N\longrightarrow L^rP/L^{r+1}P.
\]
Consider
\begin{equation}\label{eq:key-composite}
 Q_r\otimes_{\Z}\Zl
 \xrightarrow{q_r}L^rN/L^{r+1}N
 \xrightarrow{j_r}L^rP/L^{r+1}P
 \xrightarrow{\textcolor{black}{\chi_r}^{-1}}Q_r\otimes_{\Z}\Ql.
\end{equation}
\textcolor{black}{For $x\in L^r\Pi_g$ with class $\bar x\in Q_r$,
compatibility of the completion maps gives
\[
 j_rq_r(\bar x\otimes1)
 =\rho(x)L^{r+1}P
 =\textcolor{black}{\chi_r}(\bar x\otimes1).
\]
Thus the composite in \eqref{eq:key-composite} agrees
with the scalar inclusion on the dense image of $Q_r$,
hence everywhere by continuity.
This inclusion is injective because $Q_r$ is free abelian.
Since $q_r$ is surjective, it is an isomorphism and
$j_r$ is injective. We obtain}
\begin{equation}\label{eq:graded-injection}
 L^rN/L^{r+1}N
 \cong Q_r\otimes_{\Z}\Zl
 \hookrightarrow Q_r\otimes_{\Z}\Ql
 \cong L^rP/L^{r+1}P.
\end{equation}

\textcolor{black}{By \eqref{eq:filtration-inclusion} and
injectivity of $j_r$, the equality $i^{-1}(L^rP)=L^rN$
implies $i^{-1}(L^{r+1}P)=L^{r+1}N$.
Since the case $r=1$ is immediate, the equality holds
for every $r\geq1$.}

Finally, $\bigcap_{r\geq1}L^rN=\{1\}$.
Indeed, every nontrivial element $x\in N$ maps to a nontrivial
element of some finite $\ell$-group quotient $F$.
Since $F$ is nilpotent, $L^mF=\{1\}$ for some $m\geq1$.
The quotient map kills $L^mN$, so $x\notin L^mN$.
An element of $\ker i$ belongs to $i^{-1}(L^rP)=L^rN$
for every $r$ and is therefore trivial.
Thus $i$ is injective. Being a continuous injection from
a compact space into a Hausdorff space, it is a topological
embedding with closed image.
\end{proof}
\textcolor{black}{The graded comparisons preserve brackets,
so \eqref{eq:labute} gives
\begin{equation}\label{eq:Zl-presentation}
 \Gr_L^\bullet N
 \cong
 \frac{\LL_{\Zl}(A_1,B_1,\ldots,A_g,B_g)}{(\theta)}.
\end{equation}
This is also the case $(n,r)=(0,1)$ of
\cite[\S2, (2.8), pp.~201--202]{NakamuraTakaoUeno},
where all generators have degree one and the closed
weight filtration is the ordinary closed lower central
series.}

\begin{remark}
The argument applies to every prime $\ell$, including
$\ell=2$, and to every $r\geq1$.
The filtration is the ordinary closed lower central
series, not the lower $\ell$-central or Zassenhaus series.
Logarithms and exponential maps are used over $\Ql$;
no assertion that $\log N$ is a $\Zl$-Lie algebra is needed.
The integral injectivity follows from
\eqref{eq:key-composite}, before the intersection formula
is deduced.
\end{remark}
\subsection{\textcolor{black}{The center of the graded surface Lie algebra}}
\label{subsec:graded-centerlessness}

\textcolor{black}{Asada proves that the graded surface Lie algebra
has trivial center both over $\Z$ and after reduction modulo
every prime, using an argument of Labute
\cite[\S3.2]{AsadaFiltration}.
The integral statement is also recorded by
Nakamura--Takao--Ueno
\cite[\S2, immediately after (2.8.1)]{NakamuraTakaoUeno}.
We include a proof over an arbitrary field using Bergman's
Diamond Lemma.}

\begin{lemma}[Normal words for the surface relation]
\label{lem:surface-normal-words}
Let $g\geq2$ and let $F$ be any field. The images of the words
containing no adjacent pair $A_1B_1$, including the empty word,
form an $F$-basis of
\[
 \frac{F\langle A_1,B_1,\ldots,A_g,B_g\rangle}
 {\left(\sum_{j=1}^g(A_jB_j-B_jA_j)\right)},
\]
where the denominator denotes a two-sided ideal.
We call these words normal.
\end{lemma}

\begin{proof}
Order the letters with $A_1$ largest and use the
degree-lexicographic order on words. The defining relation gives
the reduction rule
\[
 A_1B_1\longmapsto
 B_1A_1-\sum_{j=2}^g(A_jB_j-B_jA_j).
\]
\textcolor{black}{Each word on the right is smaller than
$A_1B_1$ in degree-lexicographic order, so reduction
terminates. There is a single rule, hence no inclusion
ambiguity; there is no overlap ambiguity because
$A_1B_1$ has no proper self-overlap.
Bergman's Diamond Lemma
\cite[\textcolor{black}{Thm.~1.2}]{BergmanDiamond}
therefore gives the asserted basis.}
\end{proof}

We use this normal-word basis to show that a central element
of the surface Lie algebra must vanish.

\begin{lemma}[\textcolor{black}{Trivial center over a field}]\label{lem:lie-center}
Let $g\geq2$ and let $F$ be any field. Put
\[
 \mathfrak s_F
 =
 \frac{\LL_F(A_1,B_1,\ldots,A_g,B_g)}{(\theta)},
 \qquad
 \theta=\sum_{j=1}^g[A_j,B_j],
\]
where $(\theta)$ denotes the Lie ideal generated by $\theta$.
Then $Z(\mathfrak s_F)=0$. In particular,
\[
 Z\bigl(\Gr_L^\bullet N\otimes_{\Zl}\Fl\bigr)=0.
\]
\end{lemma}

\begin{proof}
The universal properties of the free Lie algebra and of the
universal enveloping algebra give
\[
 U(\mathfrak s_F)
 \cong
 \frac{F\langle A_1,B_1,\ldots,A_g,B_g\rangle}
 {\left(\sum_{j=1}^g(A_jB_j-B_jA_j)\right)}.
\]
Lemma~\ref{lem:surface-normal-words} gives a basis of normal words
for this algebra. Put $X=A_2$ and $Y=B_2$.
Prepending or appending either $X$ or $Y$ to a normal word
preserves normality.

\textcolor{black}{Suppose $f\in U(\mathfrak s_F)$ satisfies
$fX=Xf$, and write $f=\sum_{d\geq0}f_d$ in homogeneous
normal form, with only finitely many $f_d$ nonzero. Since multiplication by $X$ on either side
preserves normality, each equality $f_dX=Xf_d$ holds
in the free associative algebra.
For $d>0$, comparison of initial letters gives
$f_d=Xh$, with $h$ homogeneous of degree $d-1$.
Cancelling the initial $X$ gives $hX=Xh$, so induction
yields $f_d=c_dX^d$ for some $c_d\in F$.
Thus $f\in F[X]$.}

\textcolor{black}{If $f$ also commutes with $Y$, write
$f=\sum_{d\geq0}c_dX^d$, with only finitely many nonzero
coefficients. Then
\[
 0=fY-Yf
   =\sum_{d\geq1}c_d(X^dY-YX^d).
\]
The words $X^dY$ and $YX^d$ with $d\geq1$ are normal
and pairwise distinct, so their linear independence gives
$c_d=0$ for every $d\geq1$. Hence $f\in F$.}

By the Poincar\'e--Birkhoff--Witt theorem, the canonical map
$\mathfrak s_F\to U(\mathfrak s_F)$ is injective
\cite[\textcolor{black}{Thm.}~13.1 and \textcolor{black}{Cor.}~13.3, p.~74]{EtingofPBW}.
The image of a central element commutes with both $X$ and $Y$,
and hence is a scalar by the preceding argument.
\textcolor{black}{The canonical map preserves the internal grading induced by
word length, and $\mathfrak s_F$ has only positive degrees.
Its image therefore contains no nonzero scalar.} Thus $Z(\mathfrak s_F)=0$.
This argument uses no restriction on the characteristic of $F$.

Finally, Proposition~\ref{prop:intersection} and scalar
extension of the free Lie presentation
\eqref{eq:Zl-presentation} give
\[
 \begin{aligned}
 \Gr_L^\bullet N\otimes_{\Zl}\Fl
 &\cong
 \bigoplus_{r\geq1}(Q_r\otimes_{\Z}\Fl)\\
 &\cong
 \frac{\LL_{\Fl}(A_1,B_1,\ldots,A_g,B_g)}{(\theta)}
 =
 \mathfrak s_{\Fl}.
 \end{aligned}
\]
Taking $F=\Fl$ proves the final assertion, including $\ell=2$.
\end{proof}
\subsection{Saturation of the adjoint condition}

\textcolor{black}{The saturation property below follows from
Asada's argument in \cite[\S\S2.1--2.2]{AsadaFiltration}.
We give a direct proof in the form needed here.}
Put
\[
 \Lambda_r=Q_r\otimes_{\Z}\Zl,
 \qquad
 \Lambda=\bigoplus_{r\geq1}\Lambda_r.
\]
By Proposition~\ref{prop:intersection}, we identify $\Lambda$
with $\Gr_L^\bullet N$ and
$\Lambda_r\otimes_{\Zl}\Ql$ with $\Gr_L^rP$,
compatibly with Lie brackets.
\textcolor{black}{Each $\Lambda_r$ is a free $\Zl$-module
of finite rank.}

\begin{lemma}[Saturation of the adjoint condition]
\label{lem:saturation}
Let $r\geq1$. If
$x\in\Lambda_r\otimes_{\Zl}\Ql$ satisfies
\[
 [x,\Lambda_1]\subseteq\Lambda_{r+1},
\]
then $x\in\Lambda_r$.
\end{lemma}

\begin{proof}
\textcolor{black}{Suppose $x\notin\Lambda_r$.
Since $\Lambda_r$ is free of finite rank over $\Zl$,
there is a least integer $a\geq0$ such that
$y:=\ell^a x\in\Lambda_r$.
Then $a>0$, and minimality gives $y\notin\ell\Lambda_r$.
Thus its reduction $\bar y\in\Lambda_r/\ell\Lambda_r$
is nonzero.}

The bracket hypothesis gives
\[
 [y,\Lambda_1]
 \subseteq\ell^a\Lambda_{r+1}
 \subseteq\ell\Lambda_{r+1}.
\]
\textcolor{black}{Hence $\bar y$ commutes with the degree-one
part of $\Lambda\otimes_{\Zl}\Fl$.
Since this Lie algebra is generated in degree one by
\eqref{eq:Zl-presentation}, the Jacobi identity shows that
$\bar y$ is central, contradicting Lemma~\ref{lem:lie-center}.}
\end{proof}

\subsection{Integral conjugation}

The saturation lemma allows us to correct a conjugating
element successively by elements of $N$. Closedness of
$N$ in $P$ then completes the argument.

\begin{proposition}[Integral conjugation]\label{prop:normalizer}
For $g\geq2$,
\[
 \{u\in P:u^{-1}Nu\subseteq N\}=N.
\]
In particular, $N_P(N)=N$.
\end{proposition}

\begin{proof}
Every element of $N$ conjugates $N$ onto itself.
Conversely, suppose that $u\in P$ satisfies
$u^{-1}Nu\subseteq N$.

Suppose inductively that
\[
 \textcolor{black}{u_r}\in L^rP,\qquad \textcolor{black}{u_r}^{-1}N\textcolor{black}{u_r}\subseteq N,
\]
initially with $r=1$ and $\textcolor{black}{u_1}=u$.
Let $x_r\in\Lambda_r\otimes_{\Zl}\Ql$ be the class of $\textcolor{black}{u_r}$
in $L^rP/L^{r+1}P$.
\textcolor{black}{For every $h\in N$, the induction hypothesis
and Proposition~\ref{prop:intersection} give
\[
 [\textcolor{black}{u_r}^{-1},h]\in N\cap L^{r+1}P=L^{r+1}N.
\]
Its class in $\Gr_L^{r+1}P$ is $-[x_r,\bar h]$,
where $\bar h\in\Lambda_1$ is the class of $h$.
Since $N\to\Lambda_1$ is surjective, we obtain
\[
 [x_r,\Lambda_1]\subseteq\Lambda_{r+1},
\]
and Lemma~\ref{lem:saturation} gives $x_r\in\Lambda_r$.}

Choose $n_r\in L^rN$ with class $x_r$, and put
$\textcolor{black}{u_{r+1}}=n_r^{-1}\textcolor{black}{u_r}$.
The classes of $n_r$ and $\textcolor{black}{u_r}$ agree in
$L^rP/L^{r+1}P$, so $\textcolor{black}{u_{r+1}}\in L^{r+1}P$.
Moreover,
\[
 \textcolor{black}{u_{r+1}}^{-1}N\textcolor{black}{u_{r+1}}
 =
 \textcolor{black}{u_r}^{-1}n_rNn_r^{-1}\textcolor{black}{u_r}
 =
 \textcolor{black}{u_r}^{-1}N\textcolor{black}{u_r}
 \subseteq N.
\]
Thus the construction continues.

Put $\textcolor{black}{v_r}=n_1\cdots n_r\in N$.
Iterating this recursion from $\textcolor{black}{u_1}=u$ gives
\begin{equation}\label{eq:successive-correction}
 \textcolor{black}{v_r}^{-1}u
 =n_r^{-1}\cdots n_1^{-1}u
 =\textcolor{black}{u_{r+1}}\in L^{r+1}P.
\end{equation}
\textcolor{black}{For each finite-dimensional unipotent quotient
$U_\alpha$ of $\cP$, nilpotence implies that the projection
$q_\alpha:P\to U_\alpha(\Ql)$ kills $L^{t_\alpha}P$
for some $t_\alpha\geq1$.
By \eqref{eq:successive-correction},
\[
 q_\alpha(\textcolor{black}{v_r})=q_\alpha(u)
 \qquad\text{whenever }r+1\geq t_\alpha.
\]
Thus $\textcolor{black}{v_r}\to u$ in the inverse-limit topology.
Since $\textcolor{black}{v_r}\in N$ and $N$ is closed in $P$ by
Proposition~\ref{prop:intersection}, we obtain $u\in N$.}

\end{proof}

\section{Moving markings and comparison of sections}
\label{sec:moving}

Throughout this section, $k/\Q$ is finitely generated,
$g\geq2$, $n\geq1$, and $\ell$ is a prime. Retain $\M$, $\cX$,
and the extensions of Section~\ref{sec:conventions}.
We use the comparison $\Phi:E_\ell\to E_P$ constructed
in Subsection~\ref{subsec:weighted-pullback}. It restricts
to $i:N\to P$ on kernels and induces the identity on $\Gamma$.
Moving a marking will relate its section action to inner
conjugation in $N$ and $P$.

\subsection{The moving-mark family}
\label{subsec:moving-family}

Fix a pointed curve $(C;x_1,\ldots,x_n)$ over $\bar k$
and an index $j\in\{1,\ldots,n\}$. Put
\[
 U_j=C\setminus\{x_i:i\ne j\}.
\]
The family $\pr_2:C\times U_j\to U_j$, with sections
\[
 y\longmapsto(x_i,y)\quad(i\ne j),\qquad
 y\longmapsto(y,y)\quad(i=j),
\]
defines the classifying morphism
\begin{equation}\label{eq:moving-map}
 i_j:U_j\longrightarrow \M_{\bar k},\qquad
 y\longmapsto
 (C;x_1,\ldots,x_{j-1},y,x_{j+1},\ldots,x_n).
\end{equation}
The markings are pairwise disjoint by the definition of $U_j$.
Thus the pullback of $\cX$ along $i_j$ is $C\times U_j$,
and the $j$th tautological section pulls back to the diagonal.

Use the point $x_j\in U_j$ and the corresponding pointed
curve $b\in \M_{\bar k}$ as basepoints, with $x_j\in C$
as the basepoint of the fiber. Choose compatible paths to
the basepoints of Section~\ref{sec:conventions}, and write
\[
 \Delta=\textcolor{black}{\pi_1}(C,x_j),\quad
 E=\textcolor{black}{\pi_1}(\cX,x_j),\quad
 \Gamma=\textcolor{black}{\pi_1}(\M,b),\quad
 \Gamma_j=\textcolor{black}{\pi_1}(U_j,x_j).
\]
Let $\phi_j:\Gamma_j\to\Gamma$ be induced by $i_j$
followed by $\M_{\bar k}\to \M$, and let
$\lambda_j:\Gamma_j\to\Delta$ be induced by $U_j\hookrightarrow C$.

The map $\lambda_j$ is surjective. Indeed, a connected
finite \'etale cover of $C$ is integral, and deleting the
finite inverse image of $C\setminus U_j$ leaves it connected.
The surjectivity criterion
\cite[\textcolor{black}{Lem.}~58.4.1, Tag~0BN6]{Stacks} therefore applies.
Consequently,
\[
 q_j:\Gamma_j\xrightarrow{\lambda_j}\Delta
 \twoheadrightarrow N=\Delta^{(\ell)}
\]
is also surjective.

\subsection{Compatible pullback extensions}
\label{subsec:moving-pullbacks}

Pulling back the canonical comparison along $\phi_j$ gives
\[
 \Phi_j:E_\ell\times_\Gamma\Gamma_j
 \longrightarrow E_P\times_\Gamma\Gamma_j,
 \qquad (\textcolor{black}{x},\beta)
 \longmapsto(\Phi(\textcolor{black}{x}),\beta).
\]
Both pullbacks have quotient $\Gamma_j$, and their kernels
are $N$ and $P$, respectively.

\begin{lemma}[Compatible product extensions]
\label{lem:moving-extension}
The product family gives topological isomorphisms
\begin{equation}\label{eq:moving-extension}
 \begin{aligned}
 E_\ell\times_\Gamma\Gamma_j&\cong N\times\Gamma_j,\\
 E_P\times_\Gamma\Gamma_j&\cong P\times\Gamma_j.
 \end{aligned}
\end{equation}
They induce the identity on the kernels and quotients.
Under these identifications, $\Phi_j$ is $i\times\mathrm{id}$,
so the comparison becomes
\[
\begin{tikzcd}[column sep=small]
1 \arrow[r] &
N \arrow[r] \arrow[d,"i"] &
N\times\Gamma_j \arrow[r,"\pr_2"]
  \arrow[d,"i\times\mathrm{id}"] &
\Gamma_j \arrow[r] \arrow[d,equal] & 1 \\
1 \arrow[r] &
P \arrow[r] &
P\times\Gamma_j \arrow[r,"\pr_2"] &
\Gamma_j \arrow[r] & 1.
\end{tikzcd}
\]
\end{lemma}

\begin{proof}
Since $C$ is proper over $\bar k$, the product formula
\cite[Expos\'e~X, \textcolor{black}{Cor.}~1.7]{SGA1} identifies
$\textcolor{black}{\pi_1}(C\times U_j,(x_j,x_j))$ with
$\Delta\times\Gamma_j$.
The pulled-back family induces the commutative diagram
\[
\begin{tikzcd}[column sep=small]
1 \arrow[r] &
\Delta \arrow[r] \arrow[d,equal] &
\Delta\times\Gamma_j \arrow[r,"\pr_2"]
  \arrow[d,"\psi_j"] &
\Gamma_j \arrow[r] \arrow[d,equal] & 1 \\
1 \arrow[r] &
\Delta \arrow[r] &
E\times_\Gamma\Gamma_j \arrow[r,"\pr_2"] &
\Gamma_j \arrow[r] & 1.
\end{tikzcd}
\]
Both rows are exact by Section~\ref{sec:conventions}.
The short five lemma makes $\psi_j$ an isomorphism;
it is a topological isomorphism because both groups are profinite.

Put $R_\ell=\ker(\Delta\to N)$, so $E_\ell=E/R_\ell$.
Embed $R_\ell$ in $E\times_\Gamma\Gamma_j$ by
$r\mapsto(r,1)$.
Since $\psi_j$ is the identity on $\Delta$, it induces
\[
 \overline\psi_j:
 (\Delta\times\Gamma_j)/(R_\ell\times\{1\})
 \longrightarrow (E\times_\Gamma\Gamma_j)/R_\ell.
\]
The quotient map $E\to E_\ell$ also induces
\[
 \nu_j:(E\times_\Gamma\Gamma_j)/R_\ell
 \longrightarrow E_\ell\times_\Gamma\Gamma_j,
 \qquad
 (e,\beta)R_\ell\longmapsto(eR_\ell,\beta).
\]
These maps give the diagram
\[
\begin{tikzcd}[column sep=small]
1 \arrow[r] &
N \arrow[r] \arrow[d,equal] &
(\Delta\times\Gamma_j)/(R_\ell\times\{1\})
  \arrow[r] \arrow[d,"\overline\psi_j"] &
\Gamma_j \arrow[r] \arrow[d,equal] & 1 \\
1 \arrow[r] &
N \arrow[r] \arrow[d,equal] &
(E\times_\Gamma\Gamma_j)/R_\ell
  \arrow[r] \arrow[d,"\nu_j"] &
\Gamma_j \arrow[r] \arrow[d,equal] & 1 \\
1 \arrow[r] &
N \arrow[r] &
E_\ell\times_\Gamma\Gamma_j \arrow[r] &
\Gamma_j \arrow[r] & 1.
\end{tikzcd}
\]
The first two rows are the quotient extensions, and the
third is a pullback extension. Thus all rows are exact.
The short five lemma shows that $\overline\psi_j$ and
$\nu_j$ are isomorphisms, and profiniteness makes them
topological isomorphisms.
Finally, the quotient homomorphism $\Delta\to N$ induces
\[
 (\Delta\times\Gamma_j)/(R_\ell\times\{1\})
 \xrightarrow{\sim}N\times\Gamma_j.
\]
This proves the first identification in
\eqref{eq:moving-extension}.

\textcolor{black}{Let
$\tau_j:N\times\Gamma_j\xrightarrow{\sim}
E_\ell\times_\Gamma\Gamma_j$
denote the resulting product isomorphism.}
The constant section $y\mapsto(x_j,y)$ induces
\[
 t_j:\Gamma_j\longrightarrow E_\ell\times_\Gamma\Gamma_j,
 \qquad t_j(\beta)=\tau_j(1,\beta).
\]
Its image centralizes the kernel $N$.
Set $t_{j,P}=\Phi_j\circ t_j$.
Conjugation by $t_{j,P}(\beta)$ therefore fixes $i(N)$
\textcolor{black}{pointwise. By the pullback definition of $E_P$, this action
is conjugation by the image of $t_{j,P}(\beta)$ in
$\cG_{\cX}(\Ql)$. Since $\cP$ is normal in $\cG_{\cX}$, the action
is an algebraic automorphism of $\cP$.}
Since $i(N)$ contains the Zariski-dense completion image
of $\Delta$, this automorphism is the identity on $\cP$.
Thus $t_{j,P}(\Gamma_j)$ centralizes $P$.

It follows that
\[
 \tau_{j,P}:P\times\Gamma_j
 \longrightarrow E_P\times_\Gamma\Gamma_j,
 \qquad (p,\beta)\longmapsto p\,t_{j,P}(\beta)
\]
is a homomorphism. Every element of the target has a unique
such expression, and the inverse is
\[
 \textcolor{black}{\xi}\longmapsto
 \bigl(
 \textcolor{black}{\xi}\,
 t_{j,P}(\pr_2(\textcolor{black}{\xi}))^{-1},
 \pr_2(\textcolor{black}{\xi})
 \bigr).
\]
where the first coordinate is identified with the kernel $P$.
Both maps are continuous, so $\tau_{j,P}$ is a topological
isomorphism. By construction,
\[
 \Phi_j\circ\tau_j
 =\tau_{j,P}\circ(i\times\mathrm{id}),
\]
which proves the required compatibility.
\end{proof}

\subsection{The tautological section and its inner action}
\label{subsec:moving-inner}

\textcolor{black}{The constant section used above gives the product splittings.
The tautological section $s_j$ pulls back to the diagonal. We now compute its action on the kernels.
Retain $s_j,s_{j,P}$ and their conjugation actions $\alpha_j$
from Subsection~\ref{subsec:difference-cocycles}.}

The diagonal $y\mapsto(y,y)$ has first projection
$U_j\hookrightarrow C$ and second projection the identity. \textcolor{black}{Recall from
Subsection~\ref{subsec:moving-family} that
$\Gamma_j=\pi_1(U_j,x_j)$ and that
$q_j:\Gamma_j\to N$ is the composite
$\Gamma_j\xrightarrow{\lambda_j}\Delta\twoheadrightarrow N$.}
Under the identifications of Lemma~\ref{lem:moving-extension},
the pullbacks of $s_j$ and $s_{j,P}$ are therefore
\begin{equation}\label{eq:diagonal-group}
 \begin{aligned}
 \textcolor{black}{\phi_j^*s_j:
 \Gamma_j\longrightarrow N\times\Gamma_j},
 &\qquad
 \beta\longmapsto(q_j(\beta),\beta),\\
 \textcolor{black}{\phi_j^*s_{j,P}:
 \Gamma_j\longrightarrow P\times\Gamma_j},
 &\qquad
 \beta\longmapsto(i(q_j(\beta)),\beta).
 \end{aligned}
\end{equation}

\begin{lemma}[The moving-mark action]\label{lem:inner-action}
For every $\beta\in\Gamma_j$,
\begin{equation}\label{eq:inner-moving}
 \begin{aligned}
 \alpha_j(\phi_j(\beta))&=\Ad_{q_j(\beta)}
     &&\text{on }N,\\
 \alpha_j(\phi_j(\beta))&=\Ad_{i(q_j(\beta))}
     &&\text{on }P.
 \end{aligned}
\end{equation}
Consequently, for every $h\in N$, some $\beta\in\Gamma_j$
induces $\Ad_h$ on $N$ and $\Ad_{i(h)}$ on $P$.
\end{lemma}

\begin{proof}
In either product extension, conjugation by $(a,\beta)$
on the kernel is
\[
 (a,\beta)(b,1)(a,\beta)^{-1}=(aba^{-1},1).
\]
Apply this to the two graphs in \eqref{eq:diagonal-group}.
The final assertion follows from the surjectivity of $q_j$.
\end{proof}

\subsection{Lifting unipotent conjugacy to pro-\texorpdfstring{$\ell$}{ell}
conjugacy}
\label{subsec:conjugacy-lifting}

We now combine the moving-mark action with
Proposition~\ref{prop:normalizer} to recover conjugacy in $N$.

\begin{proposition}[A singleton comparison fiber]
\label{prop:singleton}
Let $s:\Gamma\to E_\ell$ be a continuous section.
Its image $\Phi\circ s$ is conjugate to $s_{j,P}$ by an
element of $P$ if and only if $s$ is conjugate to $s_j$
by an element of $N$.
Equivalently, the comparison of section classes has the
singleton fiber $\{[s_j]\}$ over $[s_{j,P}]$.
\end{proposition}

\begin{proof}
The implication from $N$-conjugacy to $P$-conjugacy follows
from $\Phi|_N=i$.
Conversely, suppose that a fixed $u\in P$ satisfies
\[
 \Phi(s(\gamma))=u^{-1}s_{j,P}(\gamma)u
 \qquad(\gamma\in\Gamma).
\]
Since $s(\gamma)$ and $s_j(\gamma)$ have the same image
in $\Gamma$, their difference
\[
 z(\gamma)=s(\gamma)s_j(\gamma)^{-1}
\]
belongs to $N$.
Applying $\Phi$ gives
\[
 i(z(\gamma))=u^{-1}\alpha_j(\gamma)(u).
\]
For any $h\in N$, choose $\beta\in\Gamma_j$ with
$q_j(\beta)=h$.
By Lemma~\ref{lem:inner-action}, evaluation at
$\gamma=\phi_j(\beta)$ gives
\[
 i(z(\phi_j(\beta)))
 =u^{-1}i(h)u\,i(h)^{-1}.
\]
Consequently,
\[
 u^{-1}i(h)u=i\bigl(z(\phi_j(\beta))h\bigr)\in i(N).
\]
Thus $u^{-1}i(N)u\subseteq i(N)$.
Proposition~\ref{prop:normalizer}, identifying $N$ with
its image in $P$, gives $u=i(v)$ for some $v\in N$.
The compatibility of the actions then yields
\[
 i(z(\gamma))
 =i\bigl(v^{-1}\alpha_j(\gamma)(v)\bigr).
\]
Since $i$ is injective by Proposition~\ref{prop:intersection},
\[
 s(\gamma)=v^{-1}s_j(\gamma)v
 \qquad(\gamma\in\Gamma),
\]
as required.
\end{proof}

\section{Unipotent sections over the moduli stack}
\label{sec:unipotent-sections}

Throughout this section, assume $g\geq3$ and $n\geq1$.
\textcolor{black}{We compute the unipotent section classes over $\M$,
following the method of \cite[\textcolor{black}{Sec.}~10]{WatanabeHyper2019}
and \cite[\textcolor{black}{Secs.}~14 and~18]{HainRational}.
We first determine the geometric graded sections, then compute
the cohomology and obstructions governing their arithmetic lifts.}

\subsection{Geometric graded sections}
\label{subsec:geometric-graded-sections}

Retain the notation of
Subsection~\ref{subsec:weighted-curve-sequences}.
The geometric completion sequence gives the
$R=\GSp(H)$-equivariant exact sequence
\begin{equation}\label{eq:geometric-graded-extension}
 0\longrightarrow\Gr^W_\bullet\p
 \longrightarrow\Gr^W_\bullet\u_{\cX}^{\geom}
 \xrightarrow{\pi_*}\Gr^W_\bullet\u_{\M}^{\geom}
 \longrightarrow0.
\end{equation}
The tautological sections induce graded sections
$\sigma_1,\ldots,\sigma_n$. Put
\[
 \u_{T,\leq4}^{\geom}
 =\Gr^W_\bullet(\u_T^{\geom}/W_{-5}\u_T^{\geom})
 \quad(T=\M,\cX),
 \qquad
 \p_{\leq4}=\Gr^W_\bullet(\p/W_{-5}\p).
\]
The induced projection
\begin{equation}\label{eq:four-weight-geometric-projection}
 \pi_{\leq4}:\u_{\cX,\leq4}^{\geom}
 \longrightarrow\u_{\M,\leq4}^{\geom}
\end{equation}
has kernel $\p_{\leq4}$.

\begin{proposition}\label{prop:geometric-graded-sections}
The $R$-equivariant graded Lie algebra sections of
$\pi_{\leq4}$ are exactly those induced by
$\sigma_1,\ldots,\sigma_n$.
The sections of \eqref{eq:geometric-graded-extension}
are likewise exactly $\sigma_1,\ldots,\sigma_n$.
\end{proposition}

\begin{proof}
\textcolor{black}{Let $\check\theta:\Ql(1)\to\Lambda^2H$
be the dual polarization, also viewed as an element of
$(\Lambda^2H)(-1)$, and put
$\Lambda_0^3H=(\Lambda^3H)(-1)/(\check\theta\wedge H)$.}
\textcolor{black}{Let
$\sigma:\u_{\M,\leq4}^{\geom}\to\u_{\cX,\leq4}^{\geom}$
be an $R$-equivariant graded Lie algebra section of
$\pi_{\leq4}$.} By \cite[\textcolor{black}{Thm.}~9.11 and \textcolor{black}{Prop.}~10.2]{HainRational},
the weight-$-1$ map of $\sigma$ has the form
{\color{black}
\begin{equation}\label{eq:geometric-section-weight-one}
 (v;u_1,\ldots,u_n)\longmapsto
 \left(v;\sum_{i=1}^n a_i u_i,u_1,\ldots,u_n\right),
 \qquad a_i\in\Ql,
\end{equation}
}
\textcolor{black}{where $v\in\Lambda_0^3H$ and $u_i$
belongs to the $i$th marked copy $H_i$ of $H$,
identified with $H$.}
\textcolor{black}{Put
$\Lambda_0^2H=\Lambda^2H/\operatorname{im}\check\theta$.
Hain's functor $\mathfrak d$ forms the two-step graded
quotient retaining weight $-1$ and the
$\Lambda_0^2H$-isotypic component of weight $-2$.
By \cite[\textcolor{black}{Def.}~10.1 and \textcolor{black}{Prop.}~10.2]{HainRational},
its values on $\u_{\M}^{\geom}$ and $\u_{\cX}^{\geom}$ are
$\mathfrak d_{g,n}$ and $\mathfrak d_{g,n+1}$, respectively.
Applying $\mathfrak d$ to $\sigma$ therefore gives a section
$\mathfrak d(\sigma):\mathfrak d_{g,n}\to\mathfrak d_{g,n+1}$
of the natural projection
$\mathfrak d_{g,n+1}\to\mathfrak d_{g,n}$.}
His bracket calculation gives
{\color{black}
\[
 a_i^2=a_i,\qquad
 a_i a_j=0\quad(i\ne j),\qquad
 \sum_i a_i=1\quad(g\geq4);
\]
}
see \cite[\textcolor{black}{Sec.}~10.1, \textcolor{black}{Prop.}~10.4,
and \textcolor{black}{Lem.}~10.5]{HainRational}.
Thus the weight-$-1$ map is tautological when $g\geq4$.
For $g=3$, the only additional candidate has $a_i=0$
for every $i$.

\textcolor{black}{Let $\cH_{3,n/\bar k}$ be the substack
of $\M_{\bar k}$ parametrizing hyperelliptic curves of
genus three with $n$ ordered markings.
To exclude this candidate, restrict along}
\[
 \iota:\cH_{3,n/\bar k}\longrightarrow \M_{\bar k}
\]
and pull back the universal complete curve.
Denote the Lie algebras of the geometric relative
unipotent radicals of the hyperelliptic base and total
space by $\mathfrak v_{\M}^{\geom}$ and
$\mathfrak v_{\cX}^{\geom}$, respectively.
\textcolor{black}{We use the Hodge-theoretic weight
filtrations, transported to $\Ql$ by comparison.
On $\u_{\M}^{\geom}$, $\u_{\cX}^{\geom}$, and $\p$,
these agree with the weight filtrations induced by
the action of the arithmetic weighted completions: both satisfy $W_{-r}=L^r$ for
$r\geq1$, since their abelianizations are pure of
weight $-1$. This follows from the argument of
\cite[Rem.~8.5]{HainRational} and the geometric
exact sequence.}
Put
\[
 \mathfrak v_{T,\leq4}^{\geom}
 =\Gr^W_\bullet
   (\mathfrak v_T^{\geom}/W_{-5}\mathfrak v_T^{\geom})
 \qquad(T=\M,\cX).
\]
\textcolor{black}{The comparison maps preserve these
filtrations and induce} the diagram of
$\Sp(H)$-equivariant graded Lie algebras with exact rows
\cite[\textcolor{black}{Prop.}~8.1 and
\textcolor{black}{Cor.}~8.2]{LuoWatanabeBirman}:
\[
\begin{tikzcd}[column sep=small]
0 \arrow[r] &
\p_{\leq4} \arrow[r] \arrow[d,equal] &
\mathfrak v_{\cX,\leq4}^{\geom} \arrow[r] \arrow[d] &
\mathfrak v_{\M,\leq4}^{\geom}
  \arrow[l,bend right=40,
    "\textcolor{black}{\sigma_{\mathrm{hyp}}}"']
  \arrow[r] \arrow[d] & 0 \\
0 \arrow[r] &
\p_{\leq4} \arrow[r] &
\u_{\cX,\leq4}^{\geom} \arrow[r,"\pi_{\leq4}"] &
\u_{\M,\leq4}^{\geom}
  \arrow[l,bend left=40,"\textcolor{black}{\sigma}"]
  \arrow[r] & 0.
\end{tikzcd}
\]
The right square is Cartesian because the map on kernels
is the identity.
\textcolor{black}{Hence $\sigma$ induces an
$\Sp(H)$-equivariant graded section
$\sigma_{\mathrm{hyp}}$ of the upper projection.}

{\color{black}
The right vertical map is induced by $\iota$.
On weight $-1$, it is
\[
 \begin{aligned}
 \Gr^W_{-1}\mathfrak v_{\M}^{\geom}
   \cong\bigoplus_{i=1}^n H_i
 &\longrightarrow
 \Gr^W_{-1}\u_{\M}^{\geom}
   \cong\Lambda_0^3H\oplus\bigoplus_{i=1}^n H_i,\\
 (u_1,\ldots,u_n)
 &\longmapsto(0;u_1,\ldots,u_n).
 \end{aligned}
\]
Here the zero belongs to $\Lambda_0^3H$:
the unmarked hyperelliptic abelianization has weight $-2$,
and the comparison respects the marked summands
\cite[proof of \textcolor{black}{Prop.}~8.3]{WatanabeHyper2025}.
\textcolor{black}{By \eqref{eq:geometric-section-weight-one},
the induced section $\sigma_{\mathrm{hyp}}$ has
weight-$-1$ map}
\[
 (u_1,\ldots,u_n)\longmapsto
 \left(\sum_{i=1}^n a_i u_i;u_1,\ldots,u_n\right),
\]
where the first component now belongs to the fiber
$H_0=\Gr^W_{-1}\p$.

\textcolor{black}{Applied to $\sigma_{\mathrm{hyp}}$,}
the four-weight hyperelliptic calculation gives
\[
 \sum_i a_i^4=1;
\]
see \cite[proof of \textcolor{black}{Thm.}~8.4]{WatanabeHyper2019}
and its geometric formulation in
\cite[\textcolor{black}{Thm.}~9.4]{WatanabeHyper2025}.
Comparison and scalar extension from $\Q$ to $\Ql$
preserve this calculation
\cite[\textcolor{black}{Thm.}~6.3]{HainRational}
\cite[\textcolor{black}{Rem.}~5.5]{LuoWatanabeBirman}.
\textcolor{black}{If the weight-$-1$ map of $\sigma$
were the additional genus-three candidate, then every
$a_i$ would vanish.}
This contradicts the displayed relation.
}

Finally, $\Gr^W_\bullet\u_{\M}^{\geom}$ and its four-weight
quotient are generated in weight $-1$
\cite[\textcolor{black}{Thm.}~9.11 and \textcolor{black}{Rem.}~8.5]{HainRational}.
\textcolor{black}{Thus each section is uniquely determined
by its weight-$-1$ map. The $n$ remaining possibilities
in \eqref{eq:geometric-section-weight-one} are precisely
those induced by the tautological sections.}
\end{proof}

\subsection{Cohomology and cup products}
\label{subsec:moduli-cohomology-obstruction}

Put
\[
 V_r=\Gr_L^r\p=\Gr^W_{-r}\p,
 \qquad
 T_r=V_r^{\Sp(H)}.
\]
The module $T_{2t}$ is a direct sum of copies of $\Ql(t)$,
while the odd-weight invariant parts vanish.
Together with \cite[\textcolor{black}{Cor.}~9.3]{HainRational}, this gives
\begin{equation}\label{eq:low-weight-tate-vanishing}
 T_r=0\qquad
 \text{if $r$ is odd or $1\leq r<6$.}
\end{equation}
Let $\kappa_j\in H^1(\Gamma,H)$ be the
$\ell$-adic Abel--Jacobi class of the relative degree-zero
divisor class $(2g-2)s_j-\omega_{\cX/\M}$, where
$\omega_{\cX/\M}$ is the relative canonical class.
We use continuous cohomology for $\Gamma$ and $G_k$.

\textcolor{black}{The following lemma is the analogue over $\M$
of Hain's cohomology computation and cup-product injection
for the generic curve
\cite[\textcolor{black}{Prop.}~16.3 and \textcolor{black}{Cor.}~16.4]{HainRational}.}


\begin{lemma}\label{lem:moduli-cohomology}
\leavevmode
\begin{enumerate}[label=\textup{(\roman*)}]
\item
For every $s\geq1$, $H^0(\Gamma,V_s)=0$.
There are natural isomorphisms
\begin{equation}\label{eq:moduli-first-cohomology}
 \begin{aligned}
 H^1(\Gamma,H)&=\bigoplus_{j=1}^n\Ql\kappa_j,\\
 H^1(G_k,T_s)
 &\xrightarrow[\mathrm{inf}]{\sim}
 H^1(\Gamma,T_s)
 \xrightarrow{\sim}H^1(\Gamma,V_s)
 \qquad(s\geq2),
 \end{aligned}
\end{equation}
where the last map is induced by $T_s\hookrightarrow V_s$.
In particular, $H^1(\Gamma,V_s)=0$ when $s>1$ is odd.

\item
For $r\geq2$, cup product followed by the bracket
$[-,-]:H\otimes V_{2r}\to V_{2r+1}$ gives an injective
composite
\begin{equation}\label{eq:moduli-cup-injection}
 \begin{aligned}
 H^1(\Gamma,H)\otimes_{\Ql}H^1(\Gamma,V_{2r})
 &\xrightarrow{\;\smile\;}
 H^2(\Gamma,H\otimes V_{2r})\\
 &\xrightarrow{\;[-,-]_*\;}
 H^2(\Gamma,V_{2r+1}).
 \end{aligned}
\end{equation}
\end{enumerate}
\end{lemma}

\begin{proof}
Let $s\geq2$. Write
\[
 V_s=T_s\oplus V_s^{\mathrm{nt}},
\]
where $V_s^{\mathrm{nt}}$ is the sum of the nontrivial
$\Sp(H)$-isotypic components of $V_s$.
Both summands are $R$-stable, and
\[
 (V_s^{\mathrm{nt}})^{\Sp(H)}=0.
\]
For every finite-dimensional $R$-module $V$, relative
completion gives an arithmetic-equivariant identification
\[
 H^1(\textcolor{black}{\Gamma^{\geom}},V)
 \cong
 \Hom_{\Sp(H)}\bigl(H_1(\u_{\M}^{\geom}),V\bigr).
\]
The module $H_1(\u_{\M}^{\geom})$ is pure of weight $-1$
\cite[\textcolor{black}{Prop.}~8.4 and \textcolor{black}{Thm.}~9.11]{HainRational}.
Thus $H^1(\textcolor{black}{\Gamma^{\geom}},V_s^{\mathrm{nt}})$ has weight
$1-s<0$.
Its arithmetic action factors through the cyclotomic
character, whose image is infinite. Consequently,
\[
 H^1(\textcolor{black}{\Gamma^{\geom}},V_s^{\mathrm{nt}})^{G_k}=0.
\]

Apply the Hochschild--Serre spectral sequence
\[
 E_2^{p,q}
 =H^p\bigl(G_k,
   H^q(\textcolor{black}{\Gamma^{\geom}},V_s^{\mathrm{nt}})\bigr)
 \Longrightarrow H^{p+q}(\Gamma,V_s^{\mathrm{nt}}).
\]
Its five-term exact sequence begins
\[
 \begin{aligned}
 0\longrightarrow
 H^1\bigl(G_k,(V_s^{\mathrm{nt}})^{\textcolor{black}{\Gamma^{\geom}}}\bigr)
 &\longrightarrow H^1(\Gamma,V_s^{\mathrm{nt}})\\
 &\xrightarrow{\mathrm{res}}
 H^1(\textcolor{black}{\Gamma^{\geom}},V_s^{\mathrm{nt}})^{G_k}
 \longrightarrow
 H^2\bigl(G_k,(V_s^{\mathrm{nt}})^{\textcolor{black}{\Gamma^{\geom}}}\bigr).
 \end{aligned}
\]
Since $(V_s^{\mathrm{nt}})^{\textcolor{black}{\Gamma^{\geom}}}=0$,
the two outer cohomology groups vanish.
Restriction is therefore an isomorphism, and the preceding
weight calculation gives
\[
 H^1(\Gamma,V_s^{\mathrm{nt}})
 \xrightarrow[\mathrm{res}]{\sim}
 H^1(\textcolor{black}{\Gamma^{\geom}},V_s^{\mathrm{nt}})^{G_k}=0.
\]
It follows that inclusion induces an isomorphism
\[
 H^1(\Gamma,T_s)\xrightarrow{\sim}H^1(\Gamma,V_s).
\]

The module $T_s$ is geometrically trivial.
Moreover, $H_1(\u_{\M}^{\geom})$ has no trivial
$\Sp(H)$-summand, so relative completion gives
\[
 H^1(\textcolor{black}{\Gamma^{\geom}},T_s)
 \cong
 \Hom_{\Sp(H)}\bigl(H_1(\u_{\M}^{\geom}),T_s\bigr)=0.
\]
Consider the Hochschild--Serre spectral sequence
\[
 E_2^{p,q}
 =H^p\bigl(G_k,H^q(\textcolor{black}{\Gamma^{\geom}},T_s)\bigr)
 \Longrightarrow H^{p+q}(\Gamma,T_s).
\]
Since $T_s^{\textcolor{black}{\Gamma^{\geom}}}=T_s$, its five-term exact
sequence begins
\[
 \begin{aligned}
 0\longrightarrow H^1(G_k,T_s)
 &\xrightarrow{\mathrm{inf}}H^1(\Gamma,T_s)\\
 &\xrightarrow{\mathrm{res}}
 H^1(\textcolor{black}{\Gamma^{\geom}},T_s)^{G_k}
 \longrightarrow H^2(G_k,T_s).
 \end{aligned}
\]
The target of restriction is zero by the preceding
calculation. Exactness therefore shows that inflation
is an isomorphism:
\[
 H^1(G_k,T_s)
 \xrightarrow[\mathrm{inf}]{\sim}H^1(\Gamma,T_s).
\]
This proves the second line of
\eqref{eq:moduli-first-cohomology}.
The odd-degree vanishing follows from $T_s=0$ for odd $s$.

For $H$, the Zariski density of geometric monodromy in
$\Sp(H)$ gives $H^{\textcolor{black}{\Gamma^{\geom}}}=0$.
The Hochschild--Serre five-term sequence therefore identifies
\[
 H^1(\Gamma,H)
 \xrightarrow{\sim}H^1(\textcolor{black}{\Gamma^{\geom}},H)^{G_k}.
\]
On the other hand,
\[
 H^1(\textcolor{black}{\Gamma^{\geom}},H)
 \cong
 \Hom_{\Sp(H)}
 \left(\Lambda_0^3H\oplus\bigoplus_{j=1}^nH_j,H\right)
 \cong\Ql^n.
\]
The coordinate projections onto these copies are
$R$-equivariant, so the induced $G_k$-action on this
cohomology group is trivial. Thus
\[
 H^1(\Gamma,H)
 \cong H^1(\textcolor{black}{\Gamma^{\geom}},H)\cong\Ql^n.
\]
The classes $\kappa_j$ form a basis by
\cite[\textcolor{black}{Prop.}~12.1]{HainRational}.
Finally, Zariski density of $\Gamma\to R(\Ql)$ and the
negative weight of $V_s$ give $H^0(\Gamma,V_s)=0$
for every $s\geq1$.

We now prove (ii). Consider the restriction of the bracket
\[
 b_{2r}:H\otimes T_{2r}\longrightarrow V_{2r+1},
 \qquad h\otimes t\longmapsto[h,t].
\]
Since $H$ is absolutely irreducible as an $\Sp(H)$-module,
the kernel has the form $H\otimes B$ for a subspace
$B\subseteq T_{2r}$.
Every element of $B$ commutes with $H$, which generates
$\Gr_L^\bullet\p$.
Centerlessness, proved in Lemma~\ref{lem:lie-center},
therefore gives $B=0$.
Thus $b_{2r}$ is injective.
It splits as an $R$-module map, so
\[
 b_{2r,*}:H^2(\Gamma,H\otimes T_{2r})
 \lhook\joinrel\longrightarrow H^2(\Gamma,V_{2r+1})
\]
is also injective.

By (i), inclusion identifies $H^1(\Gamma,T_{2r})$
with $H^1(\Gamma,V_{2r})$.
Naturality of cup products therefore reduces
\eqref{eq:moduli-cup-injection} to the injectivity of
the ordinary cup product
\begin{equation}\label{eq:moduli-tate-cup}
 H^1(\Gamma,H)\otimes H^1(\Gamma,T_{2r})
 \xrightarrow{\;\smile\;}
 H^2(\Gamma,H\otimes T_{2r}).
\end{equation}

Consider the Hochschild--Serre spectral sequence
\[
 E_2^{p,q}
 =H^p\bigl(G_k,
   H^q(\textcolor{black}{\Gamma^{\geom}},H\otimes T_{2r})\bigr)
 \Longrightarrow H^{p+q}(\Gamma,H\otimes T_{2r}).
\]
Since $(H\otimes T_{2r})^{\textcolor{black}{\Gamma^{\geom}}}=0$,
its low-degree exact sequence gives an injection
\[
 H^1\bigl(G_k,
   H^1(\textcolor{black}{\Gamma^{\geom}},H\otimes T_{2r})\bigr)
 \lhook\joinrel\longrightarrow
 H^2(\Gamma,H\otimes T_{2r}).
\]
Using restriction for the first factor and inflation
for the second, the cup product in
\eqref{eq:moduli-tate-cup} is the composite
\[
 \begin{aligned}
 H^1(\Gamma,H)\otimes H^1(\Gamma,T_{2r})
 &\cong
 H^1(\textcolor{black}{\Gamma^{\geom}},H)\otimes H^1(G_k,T_{2r})\\
 &\cong
 H^1\bigl(G_k,
   H^1(\textcolor{black}{\Gamma^{\geom}},H)\otimes T_{2r}\bigr)\\
 &\cong
 H^1\bigl(G_k,
   H^1(\textcolor{black}{\Gamma^{\geom}},H\otimes T_{2r})\bigr)\\
 &\lhook\joinrel\longrightarrow
 H^2(\Gamma,H\otimes T_{2r}).
 \end{aligned}
\]
Here the second isomorphism uses the finite-dimensional
trivial $G_k$-module $H^1(\textcolor{black}{\Gamma^{\geom}},H)$,
and the third uses geometric triviality of $T_{2r}$.
Compatibility of Hochschild--Serre with cup products
identifies this composite with
\eqref{eq:moduli-tate-cup}.
Thus that cup product is injective.
Composing with $b_{2r,*}$ proves
\eqref{eq:moduli-cup-injection}.
\end{proof}

\subsection{Unipotent section classes}
\label{subsec:unipotent-section-classes}
We now combine the geometric classification with the
cohomology calculations to determine the unipotent section
classes. We first treat finite quotients of the fiber and
then pass to the full kernel.

For $r\geq1$, put
\begin{equation}\label{eq:arithmetic-graded-sections}
 \mathcal S_r=H^1_{\nab}(\Gamma,\cP_r)(\Ql),
 \qquad
 \cP_r=\cP/W_{-r-1}\cP.
\end{equation}
By Proposition~\ref{prop:completion-graded-section-comparison},
$\mathcal S_r$ is naturally identified with the set of
$R$-equivariant graded sections of
\[
 \Gr^W_\bullet(\g_{\cX}/W_{-r-1}\p)
 \longrightarrow\Gr^W_\bullet\g_{\M}.
\]
Write $[s_{j,r}]$ for the tautological classes and
$\rho_r:\mathcal S_{r+1}\to\mathcal S_r$ for truncation.

Hain's obstruction construction
\cite[\textcolor{black}{Thm.}~14.3 and \textcolor{black}{Sec.}~14.2]{HainRational},
applied to these section sets, gives
\[
 H^1(\g_{\M},V_{r+1})\curvearrowright\mathcal S_{r+1}
 \xrightarrow{\rho_r}\mathcal S_r
 \xrightarrow{\delta_r^{\mathrm{Lie}}}
 H^2(\g_{\M},V_{r+1}).
\]
Here a class lifts precisely when its obstruction vanishes,
and each nonempty fiber of $\rho_r$ is a torsor under
$H^1(\g_{\M},V_{r+1})$.

By \cite[\textcolor{black}{Lem.}~3.10]{HainNab} and
\eqref{eq:weighted-cohomology-comparison}, we have
\[
 H^1(\g_{\M},V_{r+1})
 \cong H^1(\u_{\M},V_{r+1})^R
 \xrightarrow{\sim}H^1(\Gamma,V_{r+1}).
\]
Transport the torsor action through these isomorphisms.

In degree two, we use the injections
\[
 H^2(\g_{\M},V_{r+1})
 \hookleftarrow
 H^2(\u_{\M},V_{r+1})^R
 \lhook\joinrel\longrightarrow
 H^2(\Gamma,V_{r+1}).
\]
The right arrow is the comparison in
\eqref{eq:weighted-cohomology-comparison}.
Under the weight-splitting identifications, the left arrow
is induced by extending $R$-invariant cochains by zero
whenever an argument lies in $\Lie R$; restriction is
a left inverse.
In Hain's obstruction construction, the $R$-equivariant
graded lift is the identity on $\Lie R$, so its bracket
defect vanishes whenever an argument lies in $\Lie R$
\cite[proof of \textcolor{black}{Prop.}~4.8]{HainNab}.
Each class $\delta_r^{\mathrm{Lie}}(s)$ therefore has a
unique preimage in $H^2(\u_{\M},V_{r+1})^R$.
Define $\delta_r(s)$ to be the image of this preimage in
$H^2(\Gamma,V_{r+1})$.
Since both arrows are injective,
$\delta_r(s)=0$ if and only if
$\delta_r^{\mathrm{Lie}}(s)=0$.
We therefore obtain
\begin{equation}\label{eq:moduli-obstruction-sequence}
 H^1(\Gamma,V_{r+1})\curvearrowright\mathcal S_{r+1}
 \xrightarrow{\rho_r}\mathcal S_r
 \xrightarrow{\delta_r}H^2(\Gamma,V_{r+1}),
\end{equation}
with
\[
 \operatorname{im}\rho_r=\delta_r^{-1}(0),
\]
and each nonempty fiber of $\rho_r$ a torsor under
\[
 H^1(\Gamma,V_{r+1})
 \cong H^1(G_k,T_{r+1})
\]
by Lemma~\ref{lem:moduli-cohomology}.
Only these assertions about section sets are needed. The geometric classification and the vanishing of
$H^1(\Gamma,V_r)$ for $2\leq r\leq5$ give the initial
cases for the lifting argument.

\begin{lemma}\label{lem:initial-unipotent-sections}
{\color{black}
For $r=4,5$,
\[
 \mathcal S_r=\{[s_{1,r}],\ldots,[s_{n,r}]\},
\]
and these classes are distinct.
}
\end{lemma}

\begin{proof}
Let $\xi\in\mathcal S_4$, and write $\xi_r$ for its
truncation to $\mathcal S_r$.
The Cartesian square in
\eqref{eq:geometric-arithmetic-comparison} identifies
the geometric extension with the pullback of the arithmetic
extension along $\g_{\M}^{\geom}\to\g_{\M}$.
This remains true after quotienting the common kernel by
$W_{-5}\p$ and taking associated gradeds, by strictness.
Thus the graded section corresponding to $\xi$ pulls back
to a section over $\Gr^W_\bullet\g_{\M}^{\geom}$.
Restricting to negative weights and truncating through
weight $-4$ gives a geometric section of $\pi_{\leq4}$.
By Proposition~\ref{prop:geometric-graded-sections},
this section is induced by $\sigma_j$ for a unique $j$.

Since $\cP_1\cong H$ is abelian, the difference between
$\xi_1$ and $[s_{j,1}]$ defines a class in $H^1(\Gamma,H)$
whose geometric restriction vanishes, as their geometric
weight-$-1$ maps agree.
Since restriction gives
$H^1(\Gamma,H)\cong H^1(\textcolor{black}{\Gamma^{\geom}},H)$
by Lemma~\ref{lem:moduli-cohomology},
we conclude that $\xi_1=[s_{j,1}]$.
For $2\leq r\leq5$, each nonempty fiber of
$\mathcal S_r\to\mathcal S_{r-1}$ is a torsor under
\[
 H^1(\Gamma,V_r)\cong H^1(G_k,T_r)=0,
\]
by Lemma~\ref{lem:moduli-cohomology} and
\eqref{eq:low-weight-tate-vanishing}.
Thus lifts are unique whenever they exist.
The tautological sections supply these lifts, giving
$\xi=[s_{j,4}]$ and the unique lift $[s_{j,5}]$.
Consequently,
\[
 \mathcal S_r=\{[s_{1,r}],\ldots,[s_{n,r}]\}
 \qquad(r=4,5).
\]
Their distinct weight-$-1$ maps distinguish these classes.
\end{proof}
To describe the possible lifts at higher truncations,
we use the tautological lift to parametrize each torsor
fiber.
For $r\geq1$, put
\[
 \mathcal F_{j,r+1}
 :=\rho_r^{-1}([s_{j,r}]).
\]
Via the isomorphism
\[
 H^1(\Gamma,V_{r+1})\cong H^1(G_k,T_{r+1}),
\]
write $t\cdot\xi$ for the torsor action on
$\mathcal F_{j,r+1}$.
The tautological lift $[s_{j,r+1}]$ determines a bijection
\[
 \vartheta_{j,r+1}:H^1(G_k,T_{r+1})
 \xrightarrow{\sim}\mathcal F_{j,r+1},
 \qquad
 t\longmapsto t\cdot[s_{j,r+1}].
\]

Put
\[
 c_j=\frac{\kappa_j}{2g-2}\in H^1(\Gamma,H).
\]
{\color{black}
Since $H$ has weight $-1$, the degree-one comparison
\eqref{eq:weighted-cohomology-comparison} and
\cite[Prop.~8.8 and Thm.~9.11]{HainRational} give
\[
 \begin{aligned}
 H^1(\Gamma,H)
 &\cong \Hom_R(\Gr^W_{-1}\u_{\M},H)\\
 &\cong \Hom_R\left(
   \Lambda_0^3H\oplus\bigoplus_{i=1}^n H_i,H
   \right),
 \end{aligned}
\]
where $H_i\cong H$ is the summand associated with
the $i$th marking.
Under these identifications, $c_j$ corresponds to
the projection
\[
 (v;u_1,\ldots,u_n)\longmapsto u_j
\]
onto $H_j$, identified with $H$
\cite[Cor.~12.6]{HainRational}.
}
We next study whether these classes lift to the next
truncation. The following lemma expresses a projection
of the obstruction as a cup product, allowing us to
apply the injectivity in
Lemma~\ref{lem:moduli-cohomology}.

\begin{lemma}[The projected obstruction]
\label{lem:projected-cup-obstruction}
{\color{black}
Let $r\geq6$ be even, and choose an $R$-equivariant retraction
\[
 \pr_r:V_{r+1}\longrightarrow H\otimes T_r
\]
of the bracket inclusion
$b_r:H\otimes T_r\hookrightarrow V_{r+1}$.
For every $t\in H^1(G_k,T_r)$,
\begin{equation}\label{eq:projected-cup-obstruction}
 \pr_{r,*}\delta_r\bigl(\vartheta_{j,r}(t)\bigr)
 =c_j\smile\inf(t),
\end{equation}
where $\inf(t)\in H^1(\Gamma,T_r)$ is the inflation of $t$.
Equivalently, the diagram
\[
\begin{tikzcd}[column sep=large,row sep=large]
H^1(G_k,T_r)
 \arrow[r,"{t\mapsto c_j\otimes\inf(t)}"]
 \arrow[d,"\vartheta_{j,r}"',"\sim"] &
H^1(\Gamma,H)\otimes H^1(\Gamma,T_r)
 \arrow[d,"\smile"]\\
\mathcal F_{j,r}
 \arrow[r,"\pr_{r,*}\delta_r"'] &
H^2(\Gamma,H\otimes T_r)
\end{tikzcd}
\]
commutes.
}
\end{lemma}

\begin{proof}
We use Hain's construction of the obstruction
\cite[\textcolor{black}{Sec.}~14.2]{HainRational}, as in the proof of
\cite[\textcolor{black}{Prop.}~18.5]{HainRational}.
Via the degree-one comparison
\[
 H^1(\Gamma,T_r)
 \cong
 \Hom_R\bigl(H_1(\Gr^W_\bullet\u_{\M}),T_r\bigr)
\]
and composition with abelianization, we identify $\inf(t)$
with the corresponding $R$-equivariant Lie algebra
$1$-cocycle
\[
 \alpha:\Gr^W_\bullet\u_{\M}\longrightarrow T_r\subseteq V_r.
\]
By construction, $\alpha$ vanishes on commutators.
Since $T_r$ has weight $-r$, it is supported in weight $-r$.
Let $\sigma_{j,r}$ be the tautological graded section
representing $[s_{j,r}]$, and let $\widehat\sigma_j$ be its
tautological lift representing $[s_{j,r+1}]$.
The section $\sigma_{j,r}+\alpha$ represents
$\vartheta_{j,r}(t)$.
Since $\alpha$ takes values in $T_r\subseteq V_r$, we use
the composite
\[
 \Gr^W_\bullet\u_{\M}
 \xrightarrow{\alpha}T_r
 \hookrightarrow V_r
 \hookrightarrow\Gr^W_\bullet(\p/W_{-r-2}\p),
\]
where the last arrow is the canonical inclusion of the
weight-$-r$ summand. Since $\Gr^W_\bullet\u_{\M}$ is the negative-weight part of
$\Gr^W_\bullet\g_{\M}$, this composite extends uniquely as
a graded map by zero on
$\Gr^W_0\g_{\M}=\operatorname{Lie}R$.
Denote the extension again by $\alpha$, and put
\[
 \widetilde\sigma=\widehat\sigma_j+\alpha.
\]
This is an $R$-equivariant graded linear section whose
reduction modulo $W_{-r-1}\p$ is $\sigma_{j,r}+\alpha$,
the section representing $\vartheta_{j,r}(t)$.
Its failure to preserve brackets determines the
obstruction to lifting $\vartheta_{j,r}(t)$ to
$\mathcal S_{r+1}$. 
For $x,y\in\Gr^W_\bullet\u_{\M}$, the bracket defect is
\begin{equation}\label{eq:graded-obstruction-defect}
 \begin{aligned}
 h_\alpha(x,y)
 &:=[\widetilde\sigma(x),\widetilde\sigma(y)]
     -\widetilde\sigma([x,y])\\
 &=[\widehat\sigma_j(x),\alpha(y)]
   -[\widehat\sigma_j(y),\alpha(x)].
 \end{aligned}
\end{equation}
Indeed, $\widehat\sigma_j$ preserves brackets,
$\alpha([x,y])=0$ because $\alpha$ factors through
abelianization, and $[\alpha(x),\alpha(y)]$ has
weight $-2r$, so vanishes modulo $W_{-r-2}\p$.
The defect vanishes when either argument has weight $0$,
by $R$-equivariance.
Since the reduction of $\widetilde\sigma$ is a Lie algebra
section, $h_\alpha$ takes values in the kernel $V_{r+1}$,
which has weight $-(r+1)$.
Since $\alpha$ is supported in weight $-r$,
\eqref{eq:graded-obstruction-defect} shows that only
pairs of weights $(-1,-r)$ and $(-r,-1)$ can contribute. 
By \cite[\textcolor{black}{Sec.}~14.2]{HainRational}, $h_\alpha$ is an
$R$-invariant $2$-cocycle whose extension by zero on
$\Lie R$ represents
$\delta_r^{\mathrm{Lie}}(\vartheta_{j,r}(t))$.
Under the comparison
\[
 H^2(\Gr^W_\bullet\u_{\M},V_{r+1})^R
 \cong H^2(\u_{\M},V_{r+1})^R
 \lhook\joinrel\longrightarrow H^2(\Gamma,V_{r+1}),
\]
whose last arrow is the degree-two injection in
\eqref{eq:weighted-cohomology-comparison}, its class maps
to $\delta_r(\vartheta_{j,r}(t))$.

We compute the projection of this cocycle.
\textcolor{black}{We also write $c_j$ for the corresponding
$R$-equivariant graded Lie algebra $1$-cocycle
$\Gr^W_\bullet\u_{\M}\to H$.
It vanishes outside weight $-1$, where it is the
projection onto $H_j\cong H$ described above.}

Let $f_j:\M\to\mathcal M_{g,1/k}$ forget all markings
except the $j$th. Since the family equipped with $s_j$
is pulled back along $f_j$, the action on
$\Gr^W_\bullet\p$ induced by $s_j$ factors through
$\Gr^W_\bullet(f_j)_*$.
In weight $-1$, this map kills $H_i$ for $i\ne j$
and sends $u_j\in H_j$ to
$u_j\wedge\check\theta\in(\Lambda^3H)(-1)$.

By \cite[\textcolor{black}{Lem.}~9.6(ii) and \textcolor{black}{Thm.}~9.7]{HainRational},
$u_j\wedge\check\theta$ acts as the inner derivation
$\operatorname{ad}_{u_j}$: the correction term
$\textcolor{black}{\langle u_j,-\rangle}\check\theta$ in Lemma~9.6(ii) vanishes
in $\Gr^W_\bullet\p$.
Evaluating this action on $v\in T_r$ gives
\[
 [\widehat\sigma_j(u_i),v]
 =\delta_{ij}[u_i,v]
 =\delta_{ij}\,b_r(u_i\otimes v)
 \qquad(u_i\in H_i,\ v\in T_r),
\]
where $u_i$ on the right is identified with an element
of $H$.
The action of the remaining summand gives an
$R$-equivariant map
\[
 \Lambda_0^3H\otimes T_r
 \longrightarrow V_{r+1}
 \xrightarrow{\pr_r}H\otimes T_r.
\]
This composite vanishes because $T_r$ is $\Sp(H)$-trivial
and $\Hom_{\Sp(H)}(\Lambda_0^3H,H)=0$.
Also, $c_j(x)=0$ for $x\in\Lambda_0^3H$, since $c_j$
is projection onto the marked summand $H_j$.
Combining this with the calculation on the marked summands
and $\pr_r\circ b_r=\operatorname{id}$, we obtain
\[
 \pr_r[\widehat\sigma_j(x),v]=c_j(x)\otimes v
 \qquad
 (x\in\Gr^W_{-1}\u_{\M},\ v\in T_r).
\]
This identity holds for every $R$-equivariant retraction
of $b_r$.

Applying this identity to \eqref{eq:graded-obstruction-defect}
gives
\[
 \pr_r h_\alpha(x,y)
 =c_j(x)\otimes\alpha(y)-c_j(y)\otimes\alpha(x).
\]
The right-hand side is the standard alternating formula
for the cup product of the Lie algebra $1$-cocycles
$c_j$ and $\alpha$, with values in $H\otimes T_r$.
Thus $\pr_r h_\alpha$ represents $c_j\smile\alpha$ in
Lie algebra cohomology.
Applying the comparison
\eqref{eq:weighted-cohomology-comparison}
with coefficients $H\otimes T_r$, and using its compatibility
with cup products and coefficient maps, gives
\[
 \pr_{r,*}\delta_r\bigl(\vartheta_{j,r}(t)\bigr)
 =c_j\smile\inf(t),
\]
since $\alpha$ corresponds to $\inf(t)$.
This proves the assertion.
\end{proof}

The projected obstruction formula and cup-product
injectivity now allow us to extend the initial
classification to all higher truncations by induction.
\begin{proposition}\label{prop:finite-unipotent-section-classes}
For every $r\geq4$, the maps $\vartheta_{j,r}$ induce
a natural bijection
\begin{equation}\label{eq:finite-unipotent-section-classes}
 \coprod_{j=1}^n H^1(G_k,T_r)
 \xrightarrow{\sim}\mathcal S_r,
 \qquad
 (j,t)\longmapsto\vartheta_{j,r}(t).
\end{equation}
If $r\geq6$ is even, a class lifts to $\mathcal S_{r+1}$
if and only if its parameter $t$ is zero, and this lift
is unique.
\end{proposition}

\begin{proof}
We follow the induction of
\cite[\textcolor{black}{Prop.}~18.5]{HainRational}, as in
\cite[\textcolor{black}{Prop.}~10.7]{WatanabeHyper2019}.
The cases $r=4,5$ follow from
Lemma~\ref{lem:initial-unipotent-sections}, since
$T_4=T_5=0$.

Suppose that $r\geq6$ is even and that
$\mathcal S_{r-1}$ consists of the tautological classes.
Each fiber
\[
 \mathcal F_{j,r}=\rho_{r-1}^{-1}([s_{j,r-1}])
\]
contains the tautological lift $[s_{j,r}]$ and is a torsor
under $H^1(\Gamma,V_r)\cong H^1(G_k,T_r)$.
Thus the maps $\vartheta_{j,r}$ give
{\color{black}
\[
 \mathcal S_r=\coprod_{j=1}^n\mathcal F_{j,r}
 \cong\coprod_{j=1}^n H^1(G_k,T_r).
\]
}

A class $\vartheta_{j,r}(t)$ lifts to $\mathcal S_{r+1}$
if and only if its obstruction vanishes.
By Lemma~\ref{lem:projected-cup-obstruction} and the
cup-product injectivity in Lemma~\ref{lem:moduli-cohomology},
\[
 \delta_r(\vartheta_{j,r}(t))=0
 \quad\Longrightarrow\quad
 c_j\smile\inf(t)=0
 \quad\Longrightarrow\quad
 c_j\otimes t=0.
\]
Since $c_j=\kappa_j/(2g-2)\ne0$, tensoring with $c_j$
is injective over $\Ql$, so $t=0$.

Conversely, $t=0$ gives the tautological class $[s_{j,r}]$,
which lifts to $[s_{j,r+1}]$.
This lift is unique because its fiber is a torsor under
\[
 H^1(\Gamma,V_{r+1})
 \cong H^1(G_k,T_{r+1})=0,
\]
as $r+1$ is odd.
Hence $\mathcal S_{r+1}$ consists precisely of the
tautological classes, completing the induction.

\end{proof}

The lifting criterion shows that only the tautological
classes can occur in a compatible system.
Passing to the inverse limit therefore gives the
classification for the full kernel $\cP$.
\begin{proposition}\label{prop:unipotent-section-classes}
The unipotent section classes over $\M$ are exactly
the $n$ distinct tautological classes:
\begin{equation}\label{eq:unipotent-section-classification}
 H^1_{\nab}(\Gamma,\cP)(\Ql)
 =\{[s_{1,P}],\ldots,[s_{n,P}]\}.
\end{equation}
\end{proposition}

\begin{proof}
By \eqref{eq:inverse-limit-section-classes},
\[
 H^1_{\nab}(\Gamma,\cP)(\Ql)
 \cong\varprojlim_r\mathcal S_r.
\]
Let $(\xi_r)_{r\ge1}$ be a compatible system.
By Lemma~\ref{lem:initial-unipotent-sections},
$\xi_4=[s_{j,4}]$ for a unique $j$.
For $r\ge4$, a class $\vartheta_{i,r}(t)$ truncates to
$[s_{i,4}]$, so compatibility and distinctness of the
tautological classes in $\mathcal S_4$ force $i=j$.

For even $r\ge6$, write
$\xi_r=\vartheta_{j,r}(t_r)$.
Since $\rho_r(\xi_{r+1})=\xi_r$,
Proposition~\ref{prop:finite-unipotent-section-classes}
forces $t_r=0$.
The finite classification therefore gives
$\xi_r=[s_{j,r}]$ for every $r\ge4$.
Compatibility gives the same equality for $r<4$.
Hence the compatible systems are exactly those induced
by $s_1,\ldots,s_n$.
\end{proof}

\section{Descent from the generic point}
\label{sec:descent}

Retain the notation of Section~\ref{sec:conventions}.
We compare sections over $\M$ with sections over its generic
point, for both the full profinite and geometrically pro-$\ell$
extensions. We first establish descent over a smooth base
using the characteristic-zero case of Stix's specialization
result \cite[\textcolor{black}{Prop.}~91, pp.~88--89]{StixSections}
and purity, and then pass to $\M$ by a finite \'etale level cover.
For the corresponding geometrically pro-$\ell$ descent
argument in positive characteristic, see
\cite[\textcolor{black}{Prop.~10.1 and Cor.~10.2}]{WatanabePositive}.

\subsection{Descent over a smooth base}
\label{subsec:smooth-base-descent}

\begin{proposition}\label{prop:scheme-descent}
Let $S$ be a smooth integral variety over $k$, and let
$\textcolor{black}{f:}Y\to S$ be a smooth proper geometrically connected family
of curves of genus at least two. Put $F=k(S)$.
Restriction to $\Spec F$ induces a bijection between
geometric-conjugacy classes of sections over $S$ and over $F$,
for both the full profinite and geometrically pro-$\ell$
extensions.
\end{proposition}
\begin{proof}
Choose compatible geometric basepoints. Since $S$ is normal,
the natural homomorphism
\[
 q_S:G_F\twoheadrightarrow\textcolor{black}{\pi_1}(S)
\]
is surjective \cite[Tag~0BQM]{Stacks}.
The homotopy exact sequences for $Y_F/F$ and $Y/S$
have the same geometric kernel $\textcolor{black}{\pi_1}(Y_{\bar F})$;
injectivity of the fiber maps follows from
\cite[\textcolor{black}{Lem.}~4.1, p.~716]{HainRational}.
They fit into the commutative diagram with exact rows
\[
\begin{tikzcd}[column sep=small]
1 \arrow[r] &
\textcolor{black}{\pi_1}(Y_{\bar F}) \arrow[r] \arrow[d,equal] &
\textcolor{black}{\pi_1}(Y_F) \arrow[r] \arrow[d] &
G_F \arrow[r] \arrow[d,"q_S"] & 1 \\
1 \arrow[r] &
\textcolor{black}{\pi_1}(Y_{\bar F}) \arrow[r] &
\textcolor{black}{\pi_1}(Y) \arrow[r,"\textcolor{black}{f_*}"'] &
\textcolor{black}{\pi_1}(S) \arrow[r] & 1.
\end{tikzcd}
\]
The short five lemma applied to the pullback of the bottom
row gives
\begin{equation}\label{eq:descent-generic-pullback}
 \textcolor{black}{\pi_1}(Y_F)
 \cong\textcolor{black}{\pi_1}(Y)\times_{\textcolor{black}{\pi_1}(S)}G_F.
\end{equation}
Quotienting both extensions by the kernel of the maximal
pro-$\ell$ quotient of the common geometric group gives
the analogous identification for the geometrically
pro-$\ell$ extensions.

In the full profinite case, let
\[
 s:G_F\longrightarrow\textcolor{black}{\pi_1}(Y_F)
\]
be a section, and define $\tau$ to be the composite
\[
 \tau:G_F\xrightarrow{s}\textcolor{black}{\pi_1}(Y_F)
 \longrightarrow\textcolor{black}{\pi_1}(Y).
\]
Then $\textcolor{black}{f_*}\circ\tau=q_S$, and under
\eqref{eq:descent-generic-pullback}, the section $s$ is
given by
\[
 \gamma\longmapsto(\tau(\gamma),\gamma).
\]
To show that $s$ descends to a section over $S$, it
suffices to prove that $\tau$ kills $\ker q_S$.
Indeed, $\tau$ then factors uniquely as
\[
 G_F\xrightarrow{q_S}\textcolor{black}{\pi_1}(S)
 \xrightarrow{\bar s}\textcolor{black}{\pi_1}(Y).
\]
The homomorphism $\bar s$ is continuous because $q_S$
is a quotient map of profinite groups. Moreover,
$\textcolor{black}{f_*}\circ\tau=q_S$ and the surjectivity of $q_S$
imply that $\textcolor{black}{f_*}\circ\bar s=\operatorname{id}$.
Thus $\bar s$ is a section whose pullback to $G_F$
is the original section $s$.
The same construction applies in the geometrically
pro-$\ell$ case, with the fundamental groups of $Y_F$
and $Y$ replaced by their geometrically pro-$\ell$
quotients.

For each codimension-one point $v\in S$, put
\[
 R_v=\mathcal O_{S,v}^{\mathrm h},
 \qquad L=\operatorname{Frac}(R_v).
\]
Since $S$ is smooth, $R_v$ is an excellent henselian
discrete valuation ring of equal characteristic zero,
with residue field $\kappa(v)$ finitely generated over
$\Q$. 
Choose an extension $\bar v$ of $v$ to $\bar F$,
and use it to regard $L$ as a subfield of $\bar F$.
Then $G_L$ identifies with the decomposition subgroup
$D_{\bar v}\subseteq G_F$.
Write $\iota:G_L\hookrightarrow G_F$ for the inclusion
and $I_v\subseteq G_L$ for the corresponding inertia subgroup.
The base change
\[
 Y_{R_v}:=Y\times_S\Spec R_v\longrightarrow\Spec R_v
\]
is smooth and proper, so the family has good reduction
at $v$.

First consider the full profinite case. The natural square
\[
\begin{tikzcd}[column sep=large]
\textcolor{black}{\pi_1}(Y_L) \arrow[r] \arrow[d] &
G_L \arrow[d,hook,"\iota"] \\
\textcolor{black}{\pi_1}(Y_F) \arrow[r] & G_F
\end{tikzcd}
\]
is a pullback square:
\[
 \textcolor{black}{\pi_1}(Y_L)
 \cong\textcolor{black}{\pi_1}(Y_F)\times_{G_F}G_L.
\]
Indeed, the two arithmetic fundamental-group extensions
have the same geometric kernel $\textcolor{black}{\pi_1}(Y_{\bar F})$.
Let
\[
 s_v:G_L\longrightarrow\textcolor{black}{\pi_1}(Y_L)
\]
be the pullback of $s$ along $\iota$. Thus, under this
identification,
\[
 s_v(\gamma)=(s(\iota(\gamma)),\gamma).
\]
The special-fiber inclusion
$j:Y_{\kappa(v)}\hookrightarrow Y_{R_v}$ induces an
isomorphism on fundamental groups, and the specialization
map is the composite
\[
 \mathrm{sp}:\textcolor{black}{\pi_1}(Y_L)
 \longrightarrow\textcolor{black}{\pi_1}(Y_{R_v})
 \xrightarrow{j_*^{-1}}\textcolor{black}{\pi_1}(Y_{\kappa(v)});
\]
see \cite[\textcolor{black}{Sec.}~8.1]{StixSections}.
The family $Y_{R_v}\to\Spec R_v$ satisfies the hypotheses
of \cite[\textcolor{black}{Prop.}~91, pp.~88--89]{StixSections}.
In Stix's terminology, the ramification homomorphism
of $s_v$ is
\[
 \operatorname{ram}_{s_v}
 :=(\mathrm{sp}\circ s_v)|_{I_v}.
\]
That proposition rules out nontrivial tame ramification
when the residue field is finitely generated over its
prime field. Since $\kappa(v)$ has characteristic zero,
there is no wild inertia, and hence
\[
 (\mathrm{sp}\circ s_v)(I_v)=1.
\]

The compatibility with $\tau$ is expressed by the following
commutative diagram, whose unlabeled arrows are the
natural homomorphisms:
\[
\begin{tikzcd}[column sep=1.527em,row sep=1.2em]
\textcolor{black}{\pi_1}(Y_L)
  \arrow[rr]
  \arrow[dd]
  \arrow[dr]
&& G_L
  \arrow[ll,bend right=18,"s_v"']
  \arrow[dd]
  \arrow[dr,hook,"\iota"]
& \\
& \textcolor{black}{\pi_1}(Y_F)
  \arrow[rr]
  \arrow[dd]
&& G_F
  \arrow[ll,bend right=18,"s"']
  \arrow[dd,"q_S"] \\
\textcolor{black}{\pi_1}(Y_{R_v})
  \arrow[rr]
  \arrow[dr]
&& G_{\kappa(v)}
  \arrow[dr]
& \\
& \textcolor{black}{\pi_1}(Y)
  \arrow[rr,"\textcolor{black}{f_*}"']
&& \textcolor{black}{\pi_1}(S).
\end{tikzcd}
\]
The top face is the pullback square considered above.
Here the lower back row uses the canonical identification
$\textcolor{black}{\pi_1}(\Spec R_v)\cong G_{\kappa(v)}$.
Since $j_*$ is an isomorphism, the vanishing above implies
that the composite
\[
 G_L\xrightarrow{s_v}\textcolor{black}{\pi_1}(Y_L)
 \longrightarrow\textcolor{black}{\pi_1}(Y_{R_v})
\]
is trivial on $I_v$.
Composing further with the natural map to
$\textcolor{black}{\pi_1}(Y)$ gives $\tau\circ\iota$, by commutativity
of the diagram. Therefore
\[
 \tau(I_v)=1.
\]

For a section given directly in the geometrically
pro-$\ell$ quotient, apply the local ramification argument
in \cite[\textcolor{black}{Sec.}~14.2, pp.~178--179, discussion preceding
\textcolor{black}{Prop.}~199]{StixSections} to $Y_{R_v}\to\Spec R_v$.
Since $\ell$ is invertible in $R_v$ and $\kappa(v)$ is
finitely generated over $\Q$, the composite of the local
section with specialization is again trivial on $I_v$.
The same compatibility with the map to $Y$ therefore gives
$\tau(I_v)=1$ in the geometrically pro-$\ell$ quotient
of $\textcolor{black}{\pi_1}(Y)$.

Since $S$ is regular and excellent, purity of the branch
locus gives
\begin{equation}\label{eq:descent-inertia-kernel}
 \ker q_S=
 \overline{\langle\!\langle I_v:
 v\in S^{(1)}\rangle\!\rangle}_{G_F},
\end{equation}
where $S^{(1)}$ denotes the set of codimension-one points
of $S$, and the right-hand side is the closed normal
subgroup of $G_F$ generated by their inertia groups
\cite[Tag~0BMB]{Stacks}.
In either case, $\ker\tau$ is a closed normal subgroup
containing every $I_v$, so it contains $\ker q_S$.
Thus $s$ descends to a continuous section $\bar s$
over $S$, as explained above.

Finally, the geometric kernel is unchanged by restriction.
If two sections over $S$ become geometrically conjugate
over $F$, surjectivity of $q_S$ shows that the same
geometric element already conjugates them over $S$.
Conversely, geometric conjugacy over $S$ is preserved
by restriction. This proves both bijections.
\end{proof}

\subsection{Passage to the moduli stack and the generic curve}
\label{subsec:stack-descent}

A finite \'etale level cover reduces descent over $\M$
to Proposition~\ref{prop:scheme-descent}.

\begin{proposition}\label{prop:stack-descent}
Let $g\geq3$, $n\geq0$, and $K=k(\M)$.
Restriction to the generic point induces bijections
\[
 \Sec_{\M}(\Delta)\xrightarrow{\sim}\Sec_K(\Delta),
 \qquad
 \Sec_{\M}(N)\xrightarrow{\sim}\Sec_K(N).
\]
\end{proposition}

\begin{proof}
The generic stabilizer of $\M$ is trivial, so its generic
point is $\Spec K$. The map $q:G_K\to\Gamma$ is surjective,
since a connected finite \'etale cover of the normal integral
stack $\M$ remains connected over the generic point.
The generic extensions are the pullbacks of
\eqref{eq:profinite-extension} and \eqref{eq:proell-extension},
as in \eqref{eq:descent-generic-pullback}.

Put $k'=k(\mu_3)$ and \textcolor{black}{write
$\M[3]:=\cM_{g,n/k'}[3]$. Consider} the connected finite
\'etale level-$3$ cover
\[
 \textcolor{black}{\M[3]}\longrightarrow \M
\]
of Subsection~\ref{subsec:universal-extensions}. \textcolor{black}{Here $\M[3]$ is a smooth, geometrically
connected fine moduli scheme over $k'$.}
Set $K'=k'(\textcolor{black}{\M[3]})$.
\textcolor{black}{Put $\Gamma[3]:=\pi_1(\M[3])$.}
Let
\[
 q':G_{K'}\longrightarrow \textcolor{black}{\Gamma[3]}
\]
be the homomorphism induced by the generic point of
the level cover.
With compatible basepoints, we have the commutative
diagram with exact rows
\[
\begin{tikzcd}[column sep=small]
1 \arrow[r] &
\ker q \arrow[r] \arrow[d,equal] &
G_{K'} \arrow[r,"q'"] \arrow[d,hook] &
\textcolor{black}{\Gamma[3]}
  \arrow[r] \arrow[d,hook] & 1 \\
1 \arrow[r] &
\ker q \arrow[r] &
G_K \arrow[r,"q"'] &
\Gamma \arrow[r] & 1.
\end{tikzcd}
\]
Indeed, the right-hand square is Cartesian, so
\[
 G_{K'}
 =q^{-1}\bigl(\textcolor{black}{\Gamma[3]}\bigr).
\]
Consequently, $\ker q'=\ker q$ as subgroups of $G_K$.

For the full profinite extension, let
\[
 s:G_K\longrightarrow\textcolor{black}{\pi_1}(X_K)
\]
be a section, and define $\tau$ to be the composite
\[
 \tau:G_K\xrightarrow{s}\textcolor{black}{\pi_1}(X_K)
 \longrightarrow\textcolor{black}{\pi_1}(\cX).
\]
As in Proposition~\ref{prop:scheme-descent}, it suffices
to show that $\tau$ kills $\ker q$.

Write
\[
 \textcolor{black}{\cX[3]}
 :=\cX\times_{\M}\textcolor{black}{\M[3]}
\]
for the universal complete curve over the level cover.
Let $s'$ be the pullback of $s$ to $G_{K'}$.
Proposition~\ref{prop:scheme-descent} gives a section
\[
 \bar s':\textcolor{black}{\Gamma[3]}
 \longrightarrow\textcolor{black}{\pi_1}(
   \textcolor{black}{\cX[3]})
\]
whose pullback is $s'$.
These sections fit into the commutative diagram
\[
\begin{tikzcd}[column sep=1.5em,row sep=1.5em]
\textcolor{black}{\pi_1}(X_{K'})
  \arrow[rr]
  \arrow[dd]
  \arrow[dr,hook]
&& G_{K'}
  \arrow[ll,bend right=18,"s'"']
  \arrow[dd,"q'" near end]
  \arrow[dr,hook]
& \\
& \textcolor{black}{\pi_1}(X_K)
  \arrow[rr]
  \arrow[dd]
&& G_K
  \arrow[ll,bend right=18,"s"'{pos=0.25}]
  \arrow[dd,"q"] \\
\textcolor{black}{\pi_1}(\textcolor{black}{\cX[3]})
  \arrow[rr]
  \arrow[dr,hook]
&& \textcolor{black}{\Gamma[3]}
  \arrow[ll,bend right=18,"\bar s'"'{pos=0.25}]
  \arrow[dr,hook]
& \\
& \textcolor{black}{\pi_1}(\cX)
  \arrow[rr,"\pi_*"']
&& \Gamma.
\end{tikzcd}
\]
Here $q'$ is the homomorphism introduced above, and
$\ker q'=\ker q$ by the earlier Cartesian-square argument.
The top and back faces are pullback squares:
$s'$ is the pullback of $s$ along $G_{K'}\hookrightarrow G_K$
and of $\bar s'$ along $q'$.

By commutativity, the composite
\[
 G_{K'}\xrightarrow{q'}\textcolor{black}{\Gamma[3]}
 \xrightarrow{\bar s'}
 \textcolor{black}{\pi_1}(\textcolor{black}{\cX[3]})
 \longrightarrow\textcolor{black}{\pi_1}(\cX)
\]
equals $\tau|_{G_{K'}}$.
It follows that $\tau$ kills $\ker q'=\ker q$.
Thus $\tau$ factors through a continuous section
\[
 \bar s:\Gamma\longrightarrow\textcolor{black}{\pi_1}(\cX)
\]
whose pullback to $G_K$ is $s$.

The same argument applies to the geometrically pro-$\ell$
extensions, using the corresponding descent statement
in Proposition~\ref{prop:scheme-descent}.
In both cases, surjectivity of $q$ shows that two sections
over $\M$ are geometrically conjugate if and only if
their restrictions to $G_K$ are geometrically conjugate.
\end{proof}

\textcolor{black}{Proposition~\ref{prop:stack-descent}
allows us to pass from section classifications over $\M$
to those over its generic point. After the preparations
in Section~\ref{sec:full}, we prove the pro-$\ell$ and
full profinite classifications in
Section~\ref{sec:main-proofs}.}
\section{Full profinite sections}
\label{sec:full}

Throughout this section, $k/\Q$ is finitely generated,
$g\geq3$, and $n\geq1$.
Propositions~\ref{prop:unipotent-section-classes}
and~\ref{prop:singleton} classify the geometrically
pro-$\ell$ sections over $\M$ for every prime $\ell$. We now pass to the
full profinite kernel. The argument uses simultaneous motion
of the markings, the descent of Section~\ref{sec:descent}, and
Mochizuki's open Hom theorem.

\subsection{Profinite surface groups}
\label{subsec:profinite-surface-groups}

\textcolor{black}{We recall that profinite surface groups have trivial center
and establish a cohomological criterion for their endomorphisms.}

\begin{lemma}\label{lem:profinite-center}
The full profinite surface group $\Delta$ of genus at least
two has trivial center.
\end{lemma}

\begin{proof}
By \cite[\textcolor{black}{Prop.}~1.4, p.~82]{MochizukiTamagawaConfiguration},
$\Delta$ is slim: the centralizer in $\Delta$ of every open
subgroup is trivial. Taking the open subgroup to be $\Delta$
itself gives $Z(\Delta)=1$.
\end{proof}

The following criterion follows from the cohomology of
proper curves and the Hopfian property of finitely generated
profinite groups. We include the proof to make explicit
the use of all primes.

\begin{lemma}\label{lem:all-prime-auto}
Let $F:\Delta\to\Delta$ be a continuous endomorphism.
If $F^*$ is the identity on $H^1_{\cts}(\Delta,\Fl)$
for every prime $\ell$, then $F$ is an automorphism.
\end{lemma}

\begin{proof}
Let $C$ be the smooth proper geometric curve defining $\Delta$.
For every prime $\ell$, the standard comparison
\[
 H^\bullet_{\cts}(\Delta,\Fl)
 \cong H^\bullet_{\et}(C,\Fl)
\]
is compatible with cup products and pullback; see
\cite[\textcolor{black}{Prop.}~15]{Schmidt1996}.
By \cite[\textcolor{black}{Prop.}~14.2 and \textcolor{black}{Thm.}~14.7]{MilneEtale},
we have
\[
 H^2_{\cts}(\Delta,\Fl)\cong\Fl,
\]
and cup products of degree-one classes span this group.

Since $F^*$ is the identity on $H^1_{\cts}(\Delta,\Fl)$,
naturality of the cup product gives
\[
 F^*(\xi\smile\eta)=F^*\xi\smile F^*\eta=\xi\smile\eta
 \qquad
 \bigl(\xi,\eta\in H^1_{\cts}(\Delta,\Fl)\bigr).
\]
Thus $F^*$ is the identity on
$H^2_{\cts}(\Delta,\Fl)$ for every prime $\ell$.

Since $F$ is continuous and $\Delta$ is compact Hausdorff,
$F(\Delta)$ is closed. If it were proper, choose
$x\in\Delta\setminus F(\Delta)$. This complement is open,
and cosets of open normal subgroups form a neighborhood
basis in $\Delta$. Thus there is an open normal subgroup
$V\triangleleft\Delta$ such that
$xV\cap F(\Delta)=\varnothing$.
For the finite quotient $q:\Delta\to\Delta/V$, we have
$q(x)\notin q(F(\Delta))$, since equality $q(x)=q(h)$
with $h\in F(\Delta)$ would imply $h\in xV$.
Consequently,
\[
 U=q^{-1}\bigl(q(F(\Delta))\bigr)
\]
is a subgroup containing $F(\Delta)$. It is open because
$\Delta/V$ is discrete, and proper because $x\notin U$.
Put $m=[\Delta:U]>1$ and choose a prime $\ell\mid m$.
The inclusion $U\hookrightarrow\Delta$ corresponds to a
connected finite \'etale cover of degree $m$.
Under the fundamental-class identifications, restriction
\[
 \operatorname{res}^{\Delta}_{U}:
 H^2_{\cts}(\Delta,\Fl)
 \longrightarrow H^2_{\cts}(U,\Fl)
\]
is multiplication by $m$, hence zero.
This follows from the degree formula
\cite[\textcolor{black}{Rem.}~24.2(f)]{MilneEtale}, after choosing a
trivialization $\mu_\ell\cong\Fl$ over the algebraically
closed base field and using its pullback on the cover.

Since $F(\Delta)\subset U$, we may regard $F$ as a continuous
homomorphism
\[
 F_U:\Delta\longrightarrow U,\qquad \textcolor{black}{d}\longmapsto F(\textcolor{black}{d}).
\]
The resulting factorization
\[
 F:\Delta\xrightarrow{F_U}U\hookrightarrow\Delta
\]
gives
\[
 F^*=F_U^*\circ\operatorname{res}^{\Delta}_{U}=0
 \quad\text{on }H^2_{\cts}(\Delta,\Fl),
\]
contradicting the fact that $F^*$ is the identity on this
nonzero group. Hence $F$ is surjective.

Finally, $\Delta$ is topologically finitely generated, so
there are only finitely many open normal subgroups of index
at most $r$ for each $r\geq1$. Indeed, a normal subgroup of
index $m$ is the kernel of a continuous homomorphism
$\Delta\to S_m$, and such a homomorphism is determined by
the images of finitely many topological generators.

For every open normal subgroup $W$ of $\Delta$, surjectivity
of $F$ gives
$\Delta/F^{-1}(W)\cong\Delta/W$ and
$F(F^{-1}(W))=W$. Thus inverse image preserves indices
and is injective on the finite set of open normal
subgroups of index at most $r$, so it permutes that set.
Putting
\[
 \textcolor{black}{K_r}
 =\bigcap_{\substack{W\triangleleft_{\mathrm{open}}\Delta\\
                     [\Delta:W]\leq r}}W,
\]
we obtain
$F^{-1}(\textcolor{black}{K_r})=\textcolor{black}{K_r}$,
since inverse images commute with intersections.
Hence $\ker F\subseteq\textcolor{black}{K_r}$ for every $r$.

For any $x\ne1$, a finite quotient separates $x$ from the
identity. Its kernel is an open normal subgroup $W$ with
$x\notin W$. Taking $r=[\Delta:W]$ gives
$\textcolor{black}{K_r}\subseteq W$, so
$x\notin\textcolor{black}{K_r}$. Consequently,
\[
 \ker F\subseteq\bigcap_{r\geq1}\textcolor{black}{K_r}=\{1\}.
\]
Hence $F$ is bijective. Since $\Delta$ is compact Hausdorff,
its inverse is continuous, so $F$ is a topological automorphism.
\end{proof}

\subsection{Simultaneous motion and arithmetic rigidity}
\label{subsec:simultaneous-markings}

We pull a section back to the configuration space of a
fixed curve and extend it across the collision diagonals.
Restriction to the diagonal will then give the surjectivity
needed for Mochizuki's arithmetic comparison.

Choose a smooth proper geometrically connected genus-$g$
curve $C/k$; such curves exist by the hyperelliptic
construction in \cite[\textcolor{black}{Sec.}~2.8.1]{PoonenCurves}.
Put
\[
 V=\Conf_n(C)
   =\{(y_1,\ldots,y_n)\in C^n:
       y_i\ne y_j\text{ for }i\ne j\},
 \qquad
 \iota:V\hookrightarrow C^n.
\]
The constant family $\pr_2:C\times V\to V$ has the
pairwise disjoint sections
\[
 \textcolor{black}{s_{i,V}}:V\longrightarrow C\times V,\qquad
 (y_1,\ldots,y_n)\longmapsto
 (y_i,(y_1,\ldots,y_n)),
 \qquad 1\leq i\leq n.
\]
It therefore defines the classifying morphism
\[
 \varphi: V\longrightarrow \M,\qquad
 (y_1,\ldots,y_n)\longmapsto(C;y_1,\ldots,y_n),
\]
and a Cartesian diagram
\[
\begin{tikzcd}
 C\times V \arrow[r] \arrow[d,"\pr_2"'] &
 \cX \arrow[d,"\pi"] \\
 V \arrow[r, "\varphi"] & \M.
\end{tikzcd}
\]
Under this identification, the tautological sections
pull back to the sections $\textcolor{black}{s_{i,V}}$.

Choose compatible geometric basepoints. Transporting the
basepoints of the universal extension along geometric paths,
we identify $\Delta$ with $\textcolor{black}{\pi_1}(C_{\bar k})$.
Since $C$ is proper, the geometric product formula
\cite[Expos\'e~X, \textcolor{black}{Cor.}~1.7]{SGA1}, together with the
arithmetic exact sequences, gives
\[
 \textcolor{black}{\pi_1}(C\times V)
 \xrightarrow{\sim}
 \textcolor{black}{\pi_1}(C)\times_{G_k}\textcolor{black}{\pi_1}(V).
\]
The Cartesian square induces a homomorphism from this group
to $E\times_\Gamma\textcolor{black}{\pi_1}(V)$. It fits into the
commutative diagram with exact rows
\begin{equation}\label{eq:full-configuration-pullback}
\begin{tikzcd}[column sep=small]
1 \arrow[r] & \Delta \arrow[r] \arrow[d,equal] &
\textcolor{black}{\pi_1}(C)\times_{G_k}\textcolor{black}{\pi_1}(V)
  \arrow[r] \arrow[d,"\sim"] &
\textcolor{black}{\pi_1}(V) \arrow[r] \arrow[d,equal] & 1 \\
1 \arrow[r] & \Delta \arrow[r] &
E\times_\Gamma\textcolor{black}{\pi_1}(V) \arrow[r] &
\textcolor{black}{\pi_1}(V) \arrow[r] & 1.
\end{tikzcd}
\end{equation}
The maps on the kernel and quotient are identities, so the
middle map is an isomorphism by the short five lemma.

Fix a section $s:\Gamma\to E$. Its pullback $s_V$ fits into
the Cartesian diagram
\[
\begin{tikzcd}[column sep=large, row sep=large]
 \textcolor{black}{\pi_1}(C)\times_{G_k}\textcolor{black}{\pi_1}(V)
   \arrow[r]
   \arrow[d,"\operatorname{pr}_2"]
 &
 E \arrow[d]
 \\
 \textcolor{black}{\pi_1}(V)
   \arrow[r, "\varphi_\ast"]
   \arrow[u,bend left=35,"s_V"]
 &
 \Gamma \arrow[u,bend left=35,"s"]
\end{tikzcd}
\]
where the horizontal maps are induced by the classifying
morphism $V\to \M$. Since
$\operatorname{pr}_2\circ s_V=\mathrm{id}$, the section
$s_V$ is the graph
\[
 \textcolor{black}{\gamma_V}\longmapsto(f_V(\textcolor{black}{\gamma_V}),\textcolor{black}{\gamma_V})
\]
of the continuous homomorphism
\[
 f_V:=\operatorname{pr}_1\circ s_V:
 \textcolor{black}{\pi_1}(V)\longrightarrow\textcolor{black}{\pi_1}(C)
\]
over $G_k$.
For $s=s_i$, this homomorphism is induced by
the $i$th coordinate projection.

\begin{lemma}\label{lem:collision-extension}
The homomorphism $f_V$ factors uniquely through a continuous
homomorphism
\[
 f:\textcolor{black}{\pi_1}(C^n)\longrightarrow\textcolor{black}{\pi_1}(C)
\]
over $G_k$.
\end{lemma}

\begin{proof}
The constant families over $V$ and $C^n$ have the same
generic fiber. Applying Proposition~\ref{prop:scheme-descent}
to both bases extends the section over $V$ to one over
$C^n$, up to geometric conjugacy. Conjugating the extending
representative makes its restriction equal to the given
section. Under the arithmetic product identification for $C\times C^n$,
the extending section is given by
\[
 \textcolor{black}{b}\longmapsto(f(\textcolor{black}{b}),\textcolor{black}{b})
 \in\textcolor{black}{\pi_1}(C)\times_{G_k}\textcolor{black}{\pi_1}(C^n).
\]
The two components have the same image in $G_k$, so $f$
is over $G_k$. Its agreement with the given section over
$V$ gives the commutative diagram
\[
\begin{tikzcd}[column sep=large]
 \textcolor{black}{\pi_1}(V) \arrow[r,"\iota_*"] \arrow[dr,"f_V"'] &
 \textcolor{black}{\pi_1}(C^n) \arrow[d,"f"] \\
 & \textcolor{black}{\pi_1}(C)
\end{tikzcd}
\]

The map $\iota_*$ is surjective: a connected finite
\'etale cover of the normal integral variety $C^n$
is integral and remains connected over the dense open
subset $V$. This also proves uniqueness.
\end{proof}

Only the section of the constant complete-curve extension
is extended here; the classifying map $V\to \M$ need not
extend across the collision diagonals.

Choose a geometric point $\bar x\in C(\bar k)$ and use
the basepoints $\bar x$ and $(\bar x,\ldots,\bar x)$,
transporting the preceding maps along a geometric path
in $C^n$ and its coordinate paths. For the diagonal
morphism
\[
 \delta:C\longrightarrow C^n,\qquad
 y\longmapsto(y,\ldots,y),
\]
these choices give
\[
 (\pr_i)_*\circ\delta_*
 =\operatorname{id}_{\textcolor{black}{\pi_1}(C)}
 \qquad(1\leq i\leq n).
\]

For each prime $\ell$, we retain the notation $N$ and
$R_\ell$ of Subsection~\ref{subsec:geometric-kernels},
with the dependence of $N$ on $\ell$ understood.
Following \cite[Expos\'e~XIII, \S~4.0]{SGA1}, write
$\textcolor{black}{\pi_1}(Y)'$ for the geometrically pro-$\ell$ quotient
of the arithmetic fundamental group of a geometrically
connected $k$-variety $Y$. Thus we quotient by the kernel of
\[
 \textcolor{black}{\pi_1}(Y_{\bar k})
 \longrightarrow\textcolor{black}{\pi_1}(Y_{\bar k})^{(\ell)}
\]
and retain the full quotient $G_k$. The prime notation
always refers to the chosen prime $\ell$.
Under the geometric product identification
$\textcolor{black}{\pi_1}(C_{\bar k}^{\,n})\cong\Delta^n$, we have
\[
 \textcolor{black}{\pi_1}(C)'=\textcolor{black}{\pi_1}(C)/R_\ell,\qquad
 \textcolor{black}{\pi_1}(C^n)'=\textcolor{black}{\pi_1}(C^n)/R_\ell^n.
\]
The corresponding arithmetic exact sequences are
\[
\begin{aligned}
 1&\longrightarrow N
 \longrightarrow\textcolor{black}{\pi_1}(C)'
 \longrightarrow G_k\longrightarrow1,\\
 1&\longrightarrow N^n
 \longrightarrow\textcolor{black}{\pi_1}(C^n)'
 \longrightarrow G_k\longrightarrow1.
\end{aligned}
\]

A subscript $\ell$ on an arithmetic homomorphism denotes
its induced homomorphism on these quotients.
Write $f_\ell$, $\iota_{*,\ell}$, and $(\pr_i)_{*,\ell}$
for the homomorphisms induced by $f$, $\iota_*$, and
$(\pr_i)_*$, respectively. The quotient maps are compatible
with these homomorphisms:
\[
\begin{tikzcd}[column sep=large, row sep=large]
 \textcolor{black}{\pi_1}(V)
   \arrow[r,"\iota_*"] \arrow[d,two heads] &
 \textcolor{black}{\pi_1}(C^n)
   \arrow[r,shift left=0.8ex,"f"]
   \arrow[r,shift right=0.8ex,"(\pr_i)_*"']
   \arrow[d,two heads] &
 \textcolor{black}{\pi_1}(C) \arrow[d,two heads] \\
 \textcolor{black}{\pi_1}(V)'
   \arrow[r,"\iota_{*,\ell}"'] &
 \textcolor{black}{\pi_1}(C^n)'
   \arrow[r,shift left=0.8ex,"f_\ell"]
   \arrow[r,shift right=0.8ex,"(\pr_i)_{*,\ell}"'] &
 \textcolor{black}{\pi_1}(C)'.
\end{tikzcd}
\]
Here $f\circ\iota_*=f_V$, while
$(\pr_i)_*\circ\iota_*$ is induced by the $i$th coordinate
map on $V$.

\begin{lemma}\label{lem:diagonal-auto}
For each prime $\ell$, there are an index $j(\ell)$ and
an element $\textcolor{black}{v_\ell}\in N$ such that
\begin{equation}\label{eq:coordinate-proell}
 f_\ell(\textcolor{black}{b_\ell})
 =\textcolor{black}{v_\ell}^{-1}(\pr_{j(\ell)})_{*,\ell}(\textcolor{black}{b_\ell})\textcolor{black}{v_\ell}
 \qquad\bigl(\textcolor{black}{b_\ell}\in\textcolor{black}{\pi_1}(C^n)'\bigr).
\end{equation}
The homomorphism
\[
 h=f\circ\delta_*:
 \textcolor{black}{\pi_1}(C)\longrightarrow\textcolor{black}{\pi_1}(C)
\]
is an automorphism. In particular, $f$ is surjective.
\end{lemma}

\begin{proof}
For each prime $\ell$,
Propositions~\ref{prop:unipotent-section-classes}
and~\ref{prop:singleton} provide an index
\textcolor{black}{$j(\ell)$ and an element $\textcolor{black}{v_\ell}\in N$ such that the
images of the sections in $E_\ell$ satisfy}
\[
 s(\gamma)
 =\textcolor{black}{v_\ell}^{-1}s_{j(\ell)}(\gamma)\textcolor{black}{v_\ell}
 \qquad(\gamma\in\Gamma).
\]
Their pullbacks to $V$ remain conjugate by the same element.

Quotienting the middle vertical isomorphism in
\eqref{eq:full-configuration-pullback} by
$R_\ell\subset\Delta$, embedded in the common geometric
kernel, gives
\[
 E_\ell\times_\Gamma\textcolor{black}{\pi_1}(V)
 \cong
 \textcolor{black}{\pi_1}(C)'\times_{G_k}\textcolor{black}{\pi_1}(V).
\]
The pulled-back sections are graphs, and the kernel is
embedded as $N\times\{1\}$. Conjugation by $(\textcolor{black}{v_\ell},1)$
therefore conjugates their first components by $\textcolor{black}{v_\ell}$
and fixes the second components. Under the product identification, the pullbacks of
$s_{j(\ell)}$ and $s$ are given respectively by
\[
 \textcolor{black}{\gamma_V}\longmapsto
 \bigl((\pr_{j(\ell)})_*(\iota_*(\textcolor{black}{\gamma_V})),\textcolor{black}{\gamma_V}\bigr),
 \qquad
 \textcolor{black}{\gamma_V}\longmapsto
 \bigl(f(\iota_*(\textcolor{black}{\gamma_V})),\textcolor{black}{\gamma_V}\bigr),
\]
where $\textcolor{black}{\gamma_V}\in\textcolor{black}{\pi_1}(V)$ and we have used
$f_V=f\circ\iota_*$. Their conjugacy in
the geometrically pro-$\ell$ quotient therefore gives
\[
 f_\ell\bigl(\iota_{*,\ell}(\textcolor{black}{\gamma_{V,\ell}})\bigr)
 =
 \textcolor{black}{v_\ell}^{-1}(\pr_{j(\ell)})_{*,\ell}
 \bigl(\iota_{*,\ell}(\textcolor{black}{\gamma_{V,\ell}})\bigr)\textcolor{black}{v_\ell}
 \qquad\bigl(\textcolor{black}{\gamma_{V,\ell}}\in\textcolor{black}{\pi_1}(V)'\bigr).
\]
Surjectivity of $\iota_*$ implies that of $\iota_{*,\ell}$,
which proves \eqref{eq:coordinate-proell}.

Since every coordinate projection restricts to the identity
on the diagonal, this equation gives
\[
 h_\ell(t)=\textcolor{black}{v_\ell}^{-1}t\textcolor{black}{v_\ell}
 \qquad\bigl(t\in\textcolor{black}{\pi_1}(C)'\bigr).
\]
In particular, $h_\ell|_N$ is conjugation by $\textcolor{black}{v_\ell}\in N$,
so it acts trivially on
$H^1_{\cts}(N,\Fl)=\Hom_{\cts}(N,(\Fl,+))$.
Every continuous homomorphism $\Delta\to\Fl$ factors
uniquely through the maximal pro-$\ell$ quotient $N$, giving
\[
 H^1_{\cts}(N,\Fl)\xrightarrow{\sim}
 H^1_{\cts}(\Delta,\Fl).
\]
Since $F=h|_\Delta$ induces $h_\ell|_N$ on $N$, their
pullback actions correspond under the isomorphism above.
As $(h_\ell|_N)^*$ is the identity on $H^1_{\cts}(N,\Fl)$,
so is $F^*$ on $H^1_{\cts}(\Delta,\Fl)$ for every prime $\ell$. Lemma~\ref{lem:all-prime-auto} shows that $F$
is an automorphism.

Finally, the diagram
\[
\begin{tikzcd}[column sep=small]
1 \arrow[r] & \Delta \arrow[r] \arrow[d,"F"'] &
\textcolor{black}{\pi_1}(C) \arrow[r] \arrow[d,"h"] &
G_k \arrow[r] \arrow[d,equal] & 1 \\
1 \arrow[r] & \Delta \arrow[r] &
\textcolor{black}{\pi_1}(C) \arrow[r] &
G_k \arrow[r] & 1
\end{tikzcd}
\]
and the short five lemma show that $h$ is bijective.
Its inverse is continuous because these groups are
profinite. Hence $h$ is an automorphism, and
$h=f\circ\delta_*$ implies that $f$ is surjective.
\end{proof}

The preceding lemma shows that $f$ is surjective, hence
open. We now compare it with a coordinate projection
using Mochizuki's open Hom theorem for a single prime
\cite[\textcolor{black}{Thm.}~A]{MochizukiLocal}.
Recall that a field is sub-$\ell$-adic if it embeds in
a finitely generated extension of $\Ql$.

\begin{theorem*}[Mochizuki]
Let $\ell$ be a prime and let $k$ be a sub-$\ell$-adic
field. Let $Y$ be a smooth geometrically connected
$k$-variety and let $Z$ be a hyperbolic curve over $k$.
The natural maps
\[
\begin{tikzcd}[row sep=small]
 \Hom_k^{\mathrm{dom}}(Y,Z)
   \arrow[d,"\sim"] \\
 \Hom_{G_k}^{\mathrm{open}}
   \bigl(\textcolor{black}{\pi_1}(Y),\textcolor{black}{\pi_1}(Z)\bigr)
   /\textcolor{black}{\pi_1}(Z_{\bar k})
   \arrow[d,"\sim"] \\
 \Hom_{G_k}^{\mathrm{open}}
   \bigl(\textcolor{black}{\pi_1}(Y)',\textcolor{black}{\pi_1}(Z)'\bigr)
   /\textcolor{black}{\pi_1}(Z_{\bar k})^{(\ell)}
\end{tikzcd}
\]
are bijections. Here $\Hom_k^{\mathrm{dom}}$ denotes
dominant $k$-morphisms, $\Hom_{G_k}^{\mathrm{open}}$
denotes continuous open homomorphisms over $G_k$, and
the quotients are taken under conjugation by the
indicated target geometric groups.
\end{theorem*}

In particular, two continuous open arithmetic homomorphisms
of the full fundamental groups are geometrically conjugate
if their geometrically pro-$\ell$ quotients are geometrically
conjugate. We apply this comparison to $C^n$ and $C$.

\begin{proposition}\label{prop:coordinate-full}
There are an index $j$ and an element $\textcolor{black}{w}\in\Delta$ such that
\begin{equation}\label{eq:coordinate-full}
 f(b)=\textcolor{black}{w}^{-1}(\pr_j)_*(b)\textcolor{black}{w}
 \qquad\bigl(b\in\textcolor{black}{\pi_1}(C^n)\bigr).
\end{equation}
Moreover, $j(\ell)=j$ for every prime $\ell$.
\end{proposition}

\begin{proof}
Fix a prime $\ell_0$ and put $j=j(\ell_0)$.
By Lemma~\ref{lem:diagonal-auto}, $f$ is surjective.
The homomorphism $(\pr_j)_*$ is also surjective because
it has the right inverse $\delta_*$. Both are therefore
continuous open homomorphisms over $G_k$.

The field $k$ is sub-$\ell_0$-adic. Indeed, an embedding
of a transcendence-basis field for $k/\Q$ into the
corresponding rational function field over $\Q_{\ell_0}$
extends to an embedding of $k$ into a finite extension
of that field. Thus $k$ embeds in a finitely generated
extension of $\Q_{\ell_0}$, as required by
\cite[\textcolor{black}{Def.}~15.4(i)]{MochizukiLocal}.
Moreover, $C^n$ is smooth and geometrically connected,
and $C$ is hyperbolic because $g\geq3$.

\textcolor{black}{Mochizuki's theorem therefore applies. Equation
\eqref{eq:coordinate-proell} identifies the geometrically
pro-$\ell_0$ conjugacy classes of $f$ and $(\pr_j)_*$.
The second bijection in the theorem then gives a conjugator}
$\textcolor{black}{w}\in\Delta$, proving \eqref{eq:coordinate-full}.

For any other prime $\ell$, equations
\eqref{eq:coordinate-proell} and \eqref{eq:coordinate-full}
show that $(\pr_j)_{*,\ell}$ and
$(\pr_{j(\ell)})_{*,\ell}$ are conjugate by an element
of the geometric kernel $N$.
On geometric abelianizations, these maps induce the
corresponding coordinate projections
\[
 (N^{\mathrm{ab}})^{\oplus n}
 \longrightarrow N^{\mathrm{ab}},
 \qquad
 N^{\mathrm{ab}}=N/\overline{[N,N]}
 \cong\Zl^{2g}.
\]
Distinct coordinates give distinct projections, whereas
inner conjugation acts trivially on abelianization.
Hence $j(\ell)=j$.
\end{proof}

All primes were used to establish surjectivity of $f$;
one prime then sufficed for the arithmetic comparison.

\subsection{The configuration subgroup}
\label{subsec:full-section-classes}
\textcolor{black}{We identify the image of the geometric
configuration group as a normal subgroup of $\Gamma$
and describe its action on $\Delta$.}

\textcolor{black}{Recall the configuration family $C\times_k V\to V$
and its sections $s_{j,V}$ from
Subsection~\ref{subsec:simultaneous-markings}.}
Let $c:\Spec k\to\cM_{g/k}$ classify $C/k$. The square
\begin{equation}\label{eq:configuration-fiber-square}
\begin{tikzcd}[column sep=large]
 V \arrow[r,"\varphi"] \arrow[d] &
 \M=\cM_{g,n/k} \arrow[d,"\mathrm{forg}"] \\
 \Spec k \arrow[r,"c"'] &
 \cM_{g/k}
\end{tikzcd}
\end{equation}
is Cartesian as a diagram of stacks.
To see this, let $T$ be a $k$-scheme and put
$C_T=C\times_kT$. An object of the fiber product over
$T$ consists of a marked curve $(Y/T;x_1,\ldots,x_n)$
and an isomorphism $\textcolor{black}{\psi}:Y\xrightarrow{\sim}C_T$.
Transporting the markings by $\textcolor{black}{\psi}$ gives $n$
pairwise disjoint sections of $C_T\to T$.
Their projections to $C$ define a morphism $T\to V$.
Conversely, a morphism $T\to V$ gives such markings
on $C_T$, with $\textcolor{black}{\psi}=\mathrm{id}_{C_T}$.
These constructions identify the fiber product with
$V$, and its projection to $\M$ with $\varphi$.

Consequently, $V_{\bar k}$ is the geometric fiber of
the forgetful map over the point defined by $C_{\bar k}$,
and its structural map to $\M$ is
$V_{\bar k}\to V\xrightarrow{\varphi}\M$.

Return to the geometric basepoint of $V$ and its image
in $\M$, and put
\[
 D=\operatorname{im}\bigl(
   \textcolor{black}{\pi_1}(V_{\bar k})\longrightarrow\Gamma
 \bigr).
\]
We use the notation $\alpha_j$ also for the conjugation
action on $\Delta=\ker(E\to\Gamma)$:
\[
 \alpha_j(\gamma)(d)=s_j(\gamma)d\,s_j(\gamma)^{-1}
 \qquad(\gamma\in\Gamma,\ d\in\Delta).
\]
It induces the actions used in Section~\ref{sec:moving}.
\begin{lemma}\label{lem:configuration-kernel}
The subgroup $D$ is the kernel of the forgetful homomorphism
\begin{equation}\label{eq:configuration-kernel}
 \Gamma\longrightarrow\textcolor{black}{\pi_1}(\cM_{g/k}).
\end{equation}
In particular, it is closed and normal.
Moreover, $\alpha_j(D)$ contains every inner automorphism
of $\Delta$.
\end{lemma}

\begin{proof}
Using \eqref{eq:configuration-fiber-square}, the
profinite Birman exact sequence takes the form
\[
 1\longrightarrow
 \textcolor{black}{\pi_1}(V_{\bar k})
 \longrightarrow
 \Gamma
 \longrightarrow
 \textcolor{black}{\pi_1}(\cM_{g/k})
 \longrightarrow1;
\]
see \cite[\textcolor{black}{Sec.}~8.3, first diagram and its accompanying
footnote]{HainRational}.
With our compatible basepoints, the first map is
\[
 \textcolor{black}{\pi_1}(V_{\bar k})
 \longrightarrow\textcolor{black}{\pi_1}(V)
 \xrightarrow{\varphi_*}\Gamma.
\]
Its image is $D$, so exactness identifies $D$ with
the kernel of the forgetful homomorphism.
In particular, $D$ is closed and normal.

Choose distinct points $x_1,\ldots,x_n\in C(\bar k)$.
With compatible basepoints, varying only the $j$th
marking gives the commutative diagram
\[
\begin{tikzcd}[column sep=large]
 \textcolor{black}{\pi_1}\bigl(C_{\bar k}\setminus\{x_i:i\ne j\}\bigr)
 \arrow[r] \arrow[dr,two heads,"r_j"'] &
 \textcolor{black}{\pi_1}(V_{\bar k})
 \arrow[d,"(\pr_j)_*"] \\
 & \Delta.
\end{tikzcd}
\]
The map $r_j$ is induced by the open inclusion into
$C_{\bar k}$ and is surjective. Hence $(\pr_j)_*$ is
surjective. 

Under the geometric product identification, the pullback
of $s_j$ restricts to the section
\[
 s_{j,V}^{\mathrm{geom}}:
 \textcolor{black}{\pi_1}(V_{\bar k})
 \longrightarrow
 \Delta\times\textcolor{black}{\pi_1}(V_{\bar k}),
 \qquad
 \textcolor{black}{\gamma_{V,\mathrm{geom}}}\longmapsto\bigl((\pr_j)_*(\textcolor{black}{\gamma_{V,\mathrm{geom}}}),\textcolor{black}{\gamma_{V,\mathrm{geom}}}\bigr)
\]
of the second projection.
Its conjugation action on the kernel $\Delta$ is therefore
\[
 d\longmapsto
 (\pr_j)_*(\textcolor{black}{\gamma_{V,\mathrm{geom}}})\,d\,(\pr_j)_*(\textcolor{black}{\gamma_{V,\mathrm{geom}}})^{-1}
 \qquad(d\in\Delta).
\]
Since $(\pr_j)_*$ is surjective and $D$ is the image of
$\textcolor{black}{\pi_1}(V_{\bar k})$ in $\Gamma$, the group
$\alpha_j(D)$ contains every inner automorphism of $\Delta$.
\end{proof}

\section{Proofs of the main theorems and consequences}
\label{sec:main-proofs}

\textcolor{black}{\textcolor{black}{We prove the main section classifications
and discuss their consequences for the generic curve.} Let $k/\Q$ be finitely generated,
$g\geq3$, and $n\geq1$, and retain the section conventions
of Section~\ref{sec:conventions}. Put $K=k(\M)$ and
$X_K=\cX\times_{\M}\Spec K$, and write $s_{j,K}$ for the
restriction of the $j$th tautological section.}

\subsection{Geometrically pro-\texorpdfstring{$\ell$}{ell} sections}
\label{subsec:main-proell}

\begin{theorem}\label{thm:proell-classification}
For every prime $\ell$, put $N=\Delta^{(\ell)}$.
Then
\[
 \Sec_{\M}(N)=\{[s_1],\ldots,[s_n]\},
 \qquad
 \Sec_K(N)=\{[s_{1,K}],\ldots,[s_{n,K}]\}.
\]
Each set consists of $n$ distinct classes, and restriction
to the generic point is a bijection carrying $[s_j]$ to
$[s_{j,K}]$.
\end{theorem}

\begin{proof}
\textcolor{black}{
Proposition~\ref{prop:unipotent-section-classes} gives
exactly $n$ distinct unipotent section classes, and
Proposition~\ref{prop:singleton} shows that the fiber
over each consists of the corresponding tautological
pro-$\ell$ section class. This proves the classification
over $\M$.
By Proposition~\ref{prop:stack-descent}, restriction
gives a bijection
$\Sec_{\M}(N)\xrightarrow{\sim}\Sec_K(N)$.
Since it sends $[s_j]$ to $[s_{j,K}]$, the generic
classification and the asserted compatibility follow.
}
\end{proof}

\subsection{Full profinite sections}
\label{subsec:main-profinite}

\begingroup\color{black}
\begin{theorem}\label{thm:profinite-classification}
We have
\[
 \Sec_{\M}(\Delta)=\{[s_1],\ldots,[s_n]\},
 \qquad
 \Sec_K(\Delta)=\{[s_{1,K}],\ldots,[s_{n,K}]\}.
\]
Each set consists of $n$ distinct classes, and restriction
to the generic point is a bijection carrying $[s_j]$ to
$[s_{j,K}]$.
\end{theorem}
\begin{proof}
Let $s:\Gamma\to E$ be a section, and apply the construction
of Subsection~\ref{subsec:simultaneous-markings} to $s$.
Retain the notation $C,V,\varphi,\iota,f$ from that
construction and $D$ from Lemma~\ref{lem:configuration-kernel}.
Write $\varphi_*:\pi_1(V)\to\Gamma$ for the homomorphism
induced by $\varphi:V\to \M$.
The kernel of the pulled-back extension
$E\times_\Gamma\textcolor{black}{\pi_1}(V)\to\textcolor{black}{\pi_1}(V)$
is canonically identified with $\Delta$ via
$d\mapsto(d,1)$.
Under the identification
\eqref{eq:full-configuration-pullback}, the pullbacks
of $s$ and $s_j$ are given by
\[
 s_V(\textcolor{black}{\gamma_V})=\bigl(f(\iota_*(\textcolor{black}{\gamma_V})),\textcolor{black}{\gamma_V}\bigr),
 \qquad
 s_{j,V}(\textcolor{black}{\gamma_V})
 =\bigl((\pr_j)_*(\iota_*(\textcolor{black}{\gamma_V})),\textcolor{black}{\gamma_V}\bigr),
\]
where $\iota:V\hookrightarrow C^n$.
Applying Proposition~\ref{prop:coordinate-full} to
$\iota_*(\textcolor{black}{\gamma_V})$ gives an index $j$ and an element
$\textcolor{black}{w}\in\Delta$ such that
\[
 s_V(\textcolor{black}{\gamma_V})=(\textcolor{black}{w},1)^{-1}s_{j,V}(\textcolor{black}{\gamma_V})(\textcolor{black}{w},1)
 \qquad\bigl(\textcolor{black}{\gamma_V}\in\textcolor{black}{\pi_1}(V)\bigr).
\]
Since \eqref{eq:full-configuration-pullback} is the
identity on the common kernel $\Delta$, projection
from $E\times_\Gamma\textcolor{black}{\pi_1}(V)$ to $E$ gives
\[
 s(\varphi_*(\textcolor{black}{\gamma_V}))
 =\textcolor{black}{w}^{-1}s_j(\varphi_*(\textcolor{black}{\gamma_V}))\textcolor{black}{w}
 \qquad\bigl(\textcolor{black}{\gamma_V}\in\textcolor{black}{\pi_1}(V)\bigr).
\]
Replacing $s$ by the conjugate section
$\gamma\mapsto \textcolor{black}{w}\,s(\gamma)\textcolor{black}{w}^{-1}$, we may therefore
assume that $s$ and $s_j$ agree on
$\operatorname{im}(\varphi_*)$, and hence on $D$.

The difference
\[
 z(\gamma)=s(\gamma)s_j(\gamma)^{-1}\in\Delta
\]
is a cocycle for $\alpha_j$ and is trivial on $D$.
For $\gamma\in\Gamma$ and $\textcolor{black}{\gamma_0}\in D$, the cocycle identity
and $z(\textcolor{black}{\gamma_0})=1$ give
\[
 z(\gamma\textcolor{black}{\gamma_0})
 =z(\gamma)\,\alpha_j(\gamma)\bigl(z(\textcolor{black}{\gamma_0})\bigr)
 =z(\gamma).
\]
Put $\textcolor{black}{\gamma'_0}=\gamma\textcolor{black}{\gamma_0}\gamma^{-1}$.
By normality, $\textcolor{black}{\gamma'_0}\in D$, so $z(\textcolor{black}{\gamma'_0})=1$.
Since $\textcolor{black}{\gamma'_0}\gamma=\gamma\textcolor{black}{\gamma_0}$, we obtain
\[
 z(\gamma)
 =z(\textcolor{black}{\gamma'_0}\gamma)
 =z(\textcolor{black}{\gamma'_0})\,\alpha_j(\textcolor{black}{\gamma'_0})\bigl(z(\gamma)\bigr)
 =\alpha_j(\textcolor{black}{\gamma'_0})\bigl(z(\gamma)\bigr).
\]
As $\textcolor{black}{\gamma_0}$ varies over $D$, so does $\textcolor{black}{\gamma'_0}$.
Thus $z(\gamma)$
is fixed by every automorphism in $\alpha_j(D)$.
By Lemma~\ref{lem:configuration-kernel}, this includes
every inner automorphism of $\Delta$, so
\[
 \textcolor{black}{d}\,z(\gamma)\textcolor{black}{d}^{-1}=z(\gamma)
 \qquad(\textcolor{black}{d}\in\Delta).
\]
Hence $z(\gamma)$ commutes with every element of $\Delta$.
By Lemma~\ref{lem:profinite-center}, $z(\gamma)=1$.
Hence the conjugated section is $s_j$.

The tautological classes are distinct because their images
in any geometrically pro-$\ell$ quotient are distinct by
Theorem~\ref{thm:proell-classification}.
This proves the classification over $\M$.
By Proposition~\ref{prop:stack-descent}, restriction gives
a bijection
$\Sec_{\M}(\Delta)\xrightarrow{\sim}\Sec_K(\Delta)$
carrying $[s_j]$ to $[s_{j,K}]$.
The classification over $K$ and the asserted compatibility
follow.
\end{proof}
\endgroup
For every prime $\ell$, restriction and passage to the
geometrically pro-$\ell$ quotient preserve the tautological
index, so Theorems~\ref{thm:proell-classification}
and~\ref{thm:profinite-classification} show that the diagram
in the introduction is a commutative square of bijections.

\subsection{The section conjecture for the generic curve}
\label{subsec:generic-section-conjecture}

\begin{corollary}\label{cor:generic-section-conjecture}
For every prime $\ell$, put $N=\Delta^{(\ell)}$.
The natural maps
\[
 X_K(K)\longrightarrow\Sec_K(\Delta)
 \longrightarrow\Sec_K(N)
\]
are bijections. In particular, Grothendieck's section
conjecture holds for $X_K/K$.
\end{corollary}

\begin{proof}
\textcolor{black}{Hain's theorem
\textcolor{black}{\cite[\textcolor{black}{Thm.}~1, p.~711]{HainRational}}
gives $X_K(K)=\{x_1,\ldots,x_n\}$,
where $x_j$ is the $j$th marking. The displayed maps send
$x_j$ to $[s_{j,K}]$ in each section set.
Theorems~\ref{thm:profinite-classification}
and~\ref{thm:proell-classification} show that these are exactly
the $n$ distinct classes, proving both bijections.}
\end{proof}

\begin{remark}[\textcolor{black}{Topologically ample families}]\label{rem:ample-families}
\textcolor{black}{Assume $g\geq3$ and $n\geq1$.
Let $\cM_{g,n/k}[m]$ be the geometrically connected
full symplectic level-$m$ cover, where
$m\geq1$ and $\mu_m\subset k$, with
$\cM_{g,n/k}[1]=\M$.
The section classifications above, including their generic-fiber
forms, extend to its universal curve and to its pullback along
a morphism $\phi:T\longrightarrow\cM_{g,n/k}[m]$ from a smooth geometrically connected
$k$-variety inducing an isomorphism}
{\color{black}
\[
 \phi_*:\pi_1(T_{\bar k})
 \xrightarrow{\sim}
 \pi_1(\cM_{g,n/\bar k}[m]).
\]
}
\textcolor{black}{In particular, this applies to the generic ample linear
sections of dimension at least two in the automorphism-free
locus of $\cM_{g,n/k}[m]$ \cite[\textcolor{black}{Prop.}~11.3]{HainRational}.
The families over these bases are examples of Hain's topologically
ample families \cite[\textcolor{black}{Def.}~11.1]{HainRational}.}

\textcolor{black}{Indeed, Hain's level-independence results
\cite[\textcolor{black}{Prop.}~6.1, \textcolor{black}{Thm.}~6.3, and \textcolor{black}{Prop.}~8.2]{HainRational}
identify the geometric and arithmetic completions; the curve
extensions have the same kernel $\cP$, giving the same unipotent
section classes. Moving markings preserves the level structure,
so Sections~\ref{sec:moving} and~\ref{sec:full} apply to $\cM_{g,n/k}[m]$.
The fixed curve with level structure may be chosen after a finite
extension of $k$; normality of its geometric configuration subgroup
\textcolor{black}{and the cocycle argument in the proof of
Theorem~\ref{thm:profinite-classification} give}
conjugacy over $k$. For $T$, the arithmetic and homotopy exact
sequences identify the curve extensions with those over $\cM_{g,n/k}[m]$.
The descent results of Section~\ref{sec:descent} give the generic-fiber
classifications. This assertion concerns the families specified
here, rather than every topologically ample family.}
\end{remark}

\subsection*{AI disclosure}
The authors used OpenAI's ChatGPT to assist with literature
searches, aid their mathematical review,
refine notation, and improve the exposition of their original
mathematical ideas and arguments.
The authors independently reviewed all AI-assisted suggestions
and take full responsibility for the content of the paper.
\bibliographystyle{abbrvnat}
\bibliography{references}
\end{document}